\documentclass[11pt]{amsart}
\theoremstyle{plain} \numberwithin{equation}{section}
\newtheorem{Theorem}{Theorem}
\newtheorem{Lemma}[Theorem]{Lemma}

\newtheorem{Proposition}[Theorem]{Proposition}

\newtheorem{Corollary}[Theorem]{Corollary}

\theoremstyle{remark}

\title[Fourier method]
{Fourier method for one dimensional Schr\"odinger operators with
singular periodic potentials}

\author{Plamen Djakov}

\author{Boris Mityagin}

\begin{document}

\address{Sabanci University, Orhanli,
34956 Tuzla, Istanbul, Turkey}

 \email{djakov@sabanciuniv.edu}

\address{Department of Mathematics,
The Ohio State University,
 231 West 18th Ave,
Columbus, OH 43210, USA} \email{mityagin.1@osu.edu}

\begin{abstract}
By using quasi--derivatives, we develop a Fourier method for
studying the spectral properties of one dimensional Schr\"odinger
operators with periodic singular potentials.
\end{abstract}

\maketitle

\section{Introduction}

Our goal in this paper is to develop a Fourier method for studying
the spectral properties (in particular, spectral gap asymptotics)
of the Schr\"odinger operator
\begin{equation}
\label{0.1}  L(v) y = -y^{\prime \prime} + v(x) y, \quad x \in \mathbb{R},
\end{equation}
where $v$ is a singular potential such that
\begin{equation}
\label{0.2}
v(x) = v(x+\pi), \quad v \in H^{-1}_{loc} (\mathbb{R}).
\end{equation}

 In the case where the potential $v$ is
a real $L^2 ([0,\pi])$--function, it is well known by the
Floquet--Lyapunov theory (see \cite{E,Kuch,MW,W}), that the
spectrum of $L$ is  absolutely continuous and has a band--gap
structure, i.e., it is a union of closed intervals separated by
{\em spectral gaps} $$ (-\infty, \lambda_0), \;
(\lambda^-_1,\lambda^+_1), \; (\lambda^-_2,\lambda^+_2),  \cdots,
(\lambda^-_n,\lambda^+_n), \cdots .$$ The points $(\lambda^\pm_n)
$ are defined by the spectra of (\ref{0.1}) considered on the
interval $[0,\pi],$ respectively, with periodic  (for even $n$)
and anti--periodic (for odd $n$) boundary conditions $(bc):$

(a) periodic $\quad Per^+: \quad y(\pi)= y(0), \;y^\prime (\pi)=
y^\prime (0);   $

(b) antiperiodic $\quad  Per^-: \quad y(\pi)= - y(0), \;y^\prime
(\pi)= - y^\prime (0);
$

So,  one may consider the appropriate bases in $L^2 ([0,\pi]),$
which leads to a transformation of the periodic or anti--periodic
Hill--Schr\"odinger operator into an operator acting in an
$\ell^2$--sequence space.  This makes possible to develop a
Fourier method for investigation of spectra, and especially,
spectral gap asymptotics (see \cite{KM1,KM2}, where the method has
been
 used to estimate the gap asymptotics in terms of
potential smoothness). Our papers \cite{DM3,DM5} (see also the
survey \cite{DM15}) give further development of that approach and
provide  a detailed analysis of (and extensive bibliography on)
 the intimate relationship between the smoothness of the potential $v$ and
 the decay rate of the corresponding spectral gaps
(and deviations  of Dirichlet eigenvalues)
under the assumption  $v \in L^2 ([0,\pi]). $

But now singular potentials $v \in H^{-1} $ bring a lot
of new
technical problems even in the framework of the same basic
scheme as in \cite{DM15}.

First of them is to give proper understanding of the boundary
conditions (a) and (b) or their broader interpretation and careful
definition of the corresponding operators and their domains. This
is done by using quasi--derivatives. To a great extend we follow
the approach suggested (in the context of second order o.d.e.) and
developed by A. Savchuk and A. Shkalikov  \cite{SS00,SS03} (see
also \cite{SS01,SS05,SS06})
 and
 R. Hryniv and Ya. Mykytyuk \cite{HM01} (see also
\cite{HM02}-\cite{HM066}). For specific potentials see W. N.
Everitt and A. Zettl \cite{EZ79,EZ86}.

E. Korotyaev \cite{Kor03,Kor07} follows a different approach but
it works only in the case of a {\em real} potential $v.$

It is known (e.g., see \cite{HM01},   Remark 2.3, or
Proposition~\ref{prop01} below) that every $\pi$--periodic potential
$ v \in H^{-1}_{loc} (\mathbb{R}) $ has the form $$  v = C +
Q^\prime, \quad \text{where} \;\; C= const, \;\; Q \;\; \text{is} \;
\pi -\text{periodic}, \quad Q \in L^2_{loc} (\mathbb{R}). $$
Therefore, one may introduce  the ``quasi--derivative`` $u =
y^\prime -Qy $ and replace the distribution equation $-y^{\prime
\prime} + vy = 0 $ by the following system of two linear equations
with coefficients in $L^1_{loc} (\mathbb{R})$
\begin{equation}
\label{00.1}
y^\prime = Qy + u, \quad  u^\prime = (C-Q^2)y - Qu.
\end{equation}
By  the  Existence--Uniqueness theorem for systems of linear
o.d.e.  with $L^1_{loc} (\mathbb{R})$--coefficients (e.g., see
\cite{At,Naim}), the Cauchy initial value problem for the
system~(\ref{00.1}) has, for each pair of numbers $(a,b),$ a
unique solution $(y,u) $ such that $y(0) =a, \; u(0)= b.$

Moreover, following A. Savchuk and A. Shkalikov \cite{SS00,SS03},
one may consider various boundary value problems on the interval
$[0,\pi]).$  In particular,  let us consider the periodic or
anti--periodic boundary conditions $Per^\pm, $ where

($a^*$) $\quad   Per^+: \quad y(\pi)= y(0), \;\left
( y^\prime - Qy \right ) (\pi)= \left ( y^\prime - Qy  \right )
(0). $

($b^*$)  $\quad  Per^-: \quad y(\pi)= -y(0), \;\left
( y^\prime - Qy \right ) (\pi)= - \left ( y^\prime - Qy  \right )
(0). $

R. Hryniv and Ya. Mykytyuk \cite{HM01} used also the system
(\ref{00.1}) in order to give complete analysis of the spectra of
the Schr\"odinger operator with real--valued periodic
$H^{-1}$--potentials. They showed, that as in the case of periodic
$L^2_{loc} (\mathbb{R})$--potentials, the Floquet theory for the
system (\ref{00.1}) could be used to explain that if $v$ is
real--valued, then $L(v) $ is a self--adjoint operator having
absolutely continuous spectrum with band--gap structure, and the
spectral gaps are determined by the spectra of the corresponding
Hill--Schr\"odinger operators $L_{Per^\pm} $ defined in the
appropriate domains of  $L^2 ([0,\pi]) $--functions, and
considered, respectively, with the boundary conditions ($a^*$) and
($b^*$).

In Section 2 we use the same quasi--derivative approach to define
the domains of the operators $L(v)$ for complex--valued potentials
$v,$  and to explain how their spectra are described in terms of
the corresponding Lyapunov function. From a technical point of
view, our approach is different  from the approach of R. Hryniv
and Ya. Mykytyuk \cite{HM01}: they consider only the self--adjoint
case and use a quadratic form to define the domain of $L(v),$
while we consider the non--self--adjoint case as well and use the
Floquet theory and the resolvent method (see Lemma \ref{lem001}
and Theorem \ref{thm1}).

Sections 3 and 4 contains the core results of this paper. In
Section 3 we define and study the operators $L_{Per^\pm}$ which
arise when considering the Hill--Schr\"odinger operator $L(v)$
with the adjusted boundary conditions ($a^*$)
 and ($b^*$). We meticulously explain what is the Fourier
 representation of these
 operators\footnote{Maybe it is
worth to mention that T. Kappeler and C. M\"ohr
\cite{KMo} analyze "periodic and Dirichlet eigenvalues of
Schr\"odinger operators with singular potential" but they never
tell how these operators (or boundary conditions) are defined on
the interval, i.e., in a Hilbert space $L^2 ([0,\pi]).$ At some
point they jump without any justification  or explanation into
weighted $\ell^2 $--sequence spaces (an analog of Sobolev spaces
$H^{a} $)  and consider the same sequence space operators we are
used to in the regular case, i.e., if $v \in L^2_{per}
(\mathbb{R}).$ But without formulating which Sturm--Liouville
problem is considered, what are the corresponding boundary
conditions, what is the domain of the operator, etc.,  it is not
possible to pass from a {\it non-defined} differential operator to its
Fourier representation.}
 in Proposition~\ref{prop001} and Theorem~\ref{thm001}.

 In Section 4 we use the same approach as
in Section 3 to define and study the Hill--Schr\"odinger operator
$L_{Dir} (v) $  with Dirichlet  boundary conditions $ Dir: \;\;
y(0) = y(\pi ) = 0.$ Our main result there is Theorem~\ref{thm2.2}
which gives the Fourier representation  of the operator $L_{Dir}
(v). $

In Section 5 we use the Fourier representations of the operators $
L_{Per^\pm }$ and $L_{Dir} $ to study the localization of their
spectra (see Theorem~\ref{3.1}). Of course, Theorem~\ref{3.1}
gives also a rough asymptotics of the eigenvalues $\lambda^+_n,
\lambda^-_n, \mu_n $ of these operators. But we are interested to
find the asymptotics of spectral gaps $\gamma_n =\lambda^+_n -
\lambda^-_n$ in the self--adjoint case, or the asymptotics of both
$\gamma_n$ and the deviations $ \mu_n - (\lambda^+_n
+\lambda^-_n)/2 $ in the non--self--adjoint case, etc. Our results
in that direction are presented without proofs in
Section~6.\vspace{3mm}

{\em Acknowledgment.} The authors thank Professors Rostyslav
Hryniv, Andrei Shkalikov and Vadim Tkachenko for very useful
discussions of many questions of spectral analysis of differential
operators, both related and unrelated to the main topics of this
paper.

\section{Preliminary results}

1. The operator \ref{0.1} has a second term $v y$  with
$v \in (\ref{0.2}).$ First of all, let us specify the structure of
periodic functions and distributions in $H^1_{loc} (\mathbb{R})$
and  $H^{-1}_{loc} (\mathbb{R}).$

The Sobolev space  $H^1_{loc} (\mathbb{R})$ is defined as the
space of functions $f(x) \in  L^2_{loc} (\mathbb{R})$ which are
absolutely continuous and have their derivatives $f^\prime (x) \in
L^2_{loc} (\mathbb{R}).$ Therefore, for every $T>0, $
\begin{equation}
\label{0.3}
\|f\|_{1,T}^2 = \int_{-T}^T \left ( |f(x)|^2 + |f^\prime (x)|^2 \right ) < \infty.
\end{equation}

Let $\mathcal{D} (\mathbb{R}) $ be the space of all $C^\infty$--functions
on $\mathbb{R} $ with compact support, and let
$\mathcal{D} ([-T,T]) $  be the subset of all
$\varphi \in \mathcal{D} (\mathbb{R}) $ with
$supp \,\varphi \subset [-T,T].$

By definition, $H^{-1}_{loc} (\mathbb{R}) $ is the space of
distributions $v$ on $\mathbb{R} $ such that
\begin{equation}
\label{0.4} \forall T>0 \;\; \exists C(T): \quad |\langle v,
\varphi \rangle | \leq C(T) \|\varphi \|_{1,T} \quad \forall
\varphi \in \mathcal{D} ([-T,T]).
\end{equation}
Of course, since $$ \int_{-T}^T |\varphi (x)|^2 dx \leq T^2
\int_{-T}^T |\varphi^\prime (x)|^2 dx, $$ the condition
(\ref{0.4}) is equivalent to
\begin{equation}
\label{0.5} \forall T>0 \;\; \exists C(T): \quad |\langle v,
\varphi \rangle | \leq  \tilde{C}(T) \|\varphi^\prime \|_{L^2
([-T,T])} \quad \forall \varphi \in \mathcal{D} ([-T,T]).
\end{equation}
Set
\begin{equation}
\label{0.6} \mathcal{D}_1 (\mathbb{R}) = \{ \varphi^\prime :
\; \;   \varphi \in  \mathcal{D} (\mathbb{R}) \},
\quad
\mathcal{D}_1  ([-T,T])  =  \{ \varphi^\prime :
\; \;   \varphi \in  \mathcal{D} ([-T,T]) \}
\end{equation}
and consider  the linear functional
\begin{equation}
\label{0.7} q (\varphi^\prime ) := -\langle  v, \varphi \rangle,
 \qquad \varphi^\prime \in \mathcal{D}_1 (\mathbb{R}).
\end{equation}
In view of (\ref{0.5}), for each $T>0,$  $q(\cdot)$ is a
continuous linear functional defined in the space $D_1
([-T,T])\subset L^2 ([-T,T]).$ By Riesz Representation Theorem
there exists a function $Q_T (x) \in L^2 ([-T,T])$ such that
\begin{equation}
\label{0.7a} q(\varphi^\prime ) = \int_{-T}^T Q_T (x)
\varphi^\prime (x) dx \quad
 \forall   \varphi \in  \mathcal{D} ([-T,T]).
 \end{equation}
The function $Q_T $ is uniquely determined up to an additive
constant because in $L^2 ([-T,T])$ only constants are orthogonal
to $\mathcal{D}_1 ([-T,T]).$ Therefore, one can readily see that
there is a function $Q(x) \in L^2_{loc} (\mathbb{R})  $ such that
$$ q(\varphi^\prime ) = \int_{-\infty}^\infty  Q (x)
\varphi^\prime (x) dx  \qquad
 \forall   \varphi \in  \mathcal{D} (\mathbb{R}), $$
where the function $Q $ is uniquely determined up to an additive
constant. Thus, we have $$ \langle v, \varphi  \rangle = -
q(\varphi^\prime ) = -\langle Q, \varphi^\prime  \rangle  =
\langle Q^\prime, \varphi  \rangle, $$ i.e.,
\begin{equation}
\label{0.8}
v = Q^\prime.
\end{equation}

A distribution $v \in H^{-1}_{loc} (\mathbb{R})$ is called
{\em periodic of period} $\pi$ if
\begin{equation}
\label{0.9}
\langle v, \varphi (x) \rangle = \langle v, \varphi (x-\pi) \rangle
\qquad  \forall   \varphi \in  \mathcal{D} (\mathbb{R}).
\end{equation}

L. Schwartz \cite{Sch} gave an equivalent definition of a {\em
periodic of period} $\pi$ distribution in the following way: Let
$$ \omega: \; \mathbb{R} \to S^1 = \mathbb{R}/\pi\mathbb{Z}, \quad
\omega (x) = x  \mod  \pi. $$ A distribution $ F \in
\mathcal{D}^\prime (\mathbb{R}) $ is {\em periodic } if, for some
$f \in \left (C^\infty (S^1) \right )^\prime, $ we have
$$ F(x) = f(\omega (x)), \quad \text {i.e.,} \;\;
\langle \varphi, F \rangle = \langle \Phi , f \rangle, $$
where
$$ \Phi = \sum_{k \in \mathbb{Z}} \varphi (x-k\pi ).$$

Now, if $v$ is periodic and  $ Q \in L^2_{loc} (\mathbb{R} ) $ is
chosen so that
 (\ref{0.8}) holds, we have by (\ref{0.9})
 $$ \int_{-\infty}^\infty Q(x+\pi) \varphi^\prime  (x) dx =
 \int_{-\infty}^\infty Q(x) \varphi^\prime (x-\pi) =
 \int_{-\infty}^\infty Q(x) \varphi^\prime (x) dx, $$
 i.e.,
 $$ \int_{-\infty}^\infty [Q(x+\pi ) - Q(x) ] \varphi^\prime (x) dx = 0
\qquad  \forall   \varphi \in  \mathcal{D} (\mathbb{R}). $$
Thus, there exists a constant $c $ such that
$$ Q(x+\pi ) - Q(x)  = c \quad a.e. $$
Consider the function
$$ \tilde{Q} (x) = Q(x) - \frac{c}{\pi} x ; $$
then we have
$ \tilde{Q} (x+\pi) =
\tilde{Q} (x)  $   a.e., so
$\tilde{Q} $ is  $\pi$--periodic, and
$$ v = \tilde{Q}^\prime  +  \frac{c}{\pi}.$$
Let
\begin{equation}
\label{0.11}
\tilde{Q} (x) =  \sum_{m \in 2\mathbb{Z}} q(m) e^{imx}
\end{equation}
be the Fourier series expansion of the function $\tilde{Q} \in L^2([0,\pi]). $
Set
\begin{equation}
\label{0.12}
V(0) = \frac{c}{\pi}, \quad V(m) =  im q(m) \quad \text{for} \;\; m \neq 0.
\end{equation}
All this leads to the following statement.

\begin{Proposition}
\label{prop01}
Every $\pi$--periodic distribution  $v \in H^{-1}_{loc} (\mathbb{R}) $
has the form
\begin{equation}
\label{0.13} v =  C + Q^\prime, \qquad  Q \in L^2_{loc}
(\mathbb{R}), \quad Q (x+\pi) \stackrel{a.e.}{=} Q(x)
\end{equation}
with
\begin{equation}
\label{0.13a} q(0)  = \frac{1}{\pi}
\int_0^\pi Q(x) dx = 0,
\end{equation}
and can be written as a
converging in $H_{loc}^{-1} (\mathbb{R}) $
Fourier series
\begin{equation}
\label{0.14} v =  \sum_{m \in 2\mathbb{Z}} V(m) e^{imx}
\end{equation}
with
\begin{equation}
\label{0.15} V(0) = C, \quad V(m) =  im q(m) \quad \text{for} \;\;
m \neq 0,
\end{equation}
where $q(m) $ are the Fourier coefficients of $Q.$
Of course,
\begin{equation}
\label{0.16} \|Q\|^2_{L^2([0,\pi])}  =  \sum_{m \neq 0}
\frac{|V(m)|^2}{m^2}.
\end{equation}
\end{Proposition}

{\em Remark.}  R. Hryniv and Ya. Mykytyuk \cite{HM01}, (see Theorem
3.1 and Remark 2.3) give a more general claim about  the structure
of uniformly bounded $H^{-1}_{loc}
(\mathbb{R})$--distributions.\vspace{3mm}

2.  In view of (\ref{0.4}), each distribution $v \in H^{-1}_{loc}
(\mathbb{R}) $ could be considered as a linear functional on the
space $H^1_{oo} (\mathbb{R})$ of functions in $H^1_{loc}
(\mathbb{R})$ with compact support. Therefore, if $v \in
H^{-1}_{loc} (\mathbb{R}) $ and $y \in H^1_{loc} (\mathbb{R}),$
then the differential expression $ \ell (y)= - y^{\prime \prime} +
v\cdot y $ is well--defined by $$\langle - y^{\prime \prime} +
v\cdot y, \varphi \rangle = \langle y^\prime, \varphi^\prime
\rangle + \langle v, y\cdot \varphi \rangle $$ as a distribution
in $H^{-1}_{loc} (\mathbb{R}). $ This observation suggests to
consider the Schr\"odinger operator $-d^2/dx^2 + v $ in the domain
\begin{equation}
\label{0.17} D(L(v))= \left \{y \in H^1_{loc} (\mathbb{R}) \cap
L^2 (\mathbb{R}) \;: \quad - y^{\prime \prime} + v\cdot y \in L^2
(\mathbb{R}) \right \}.
\end{equation}
Moreover, suppose $v =C + Q^\prime, $ where $C$ is a constant and
$ Q $ is a $\pi$--periodic function such that
\begin{equation}
\label{0.18} Q \in L^2 ([0,\pi]), \quad q(0) = \frac{1}{\pi}
\int_0^\pi Q(x) dx = 0.
\end{equation}
Then the differential expression $ \ell (y) = - y^{\prime \prime}
+ vy $ can be written in the form
\begin{equation}
\label{001} \ell (y) = - \left ( y^\prime - Qy  \right )^\prime -
Q y^\prime + C y.
\end{equation}
Notice that $$\ell (y) = - \left ( y^\prime - Qy  \right )^\prime
- Q y^\prime + C y = f \in L^2 (\mathbb{R}) $$ {\em if and only if
$$ u = y^\prime - Qy  \in W^1_{1,loc} (\mathbb{R}) $$ and the pair
$(y,u)$ satisfies the system of differential equations}
\begin{equation}
\label{0.31}  \left | \begin{array}{l}  y^\prime = Qy +u,\\
u^\prime = (C- Q^2) y - Qu + f. \end{array}  \right.
\end{equation}
Consider the corresponding homogeneous system
\begin{equation}
\label{0.32}  \left | \begin{array}{l}  y^\prime = Qy +u,\\
u^\prime = (C- Q^2) y - Qu. \end{array}  \right.
\end{equation}
with initial data
\begin{equation}
\label{0.33} y(0) = a, \quad u(0) = b.
\end{equation}
Since the coefficients $1, Q, C- Q^2 $  of the system (\ref{0.32})
are in $L^1_{loc}(\mathbb{R}),$ the standard existence--uniqueness
theorem for linear systems of  equations with
$L^1_{loc}(\mathbb{R})$--coefficients (e.g., see M. Naimark
\cite{Naim}, Sect.16, or F. Atkinson \cite{At}) guarantees that
for any pair of numbers $(a,b)$ the system (\ref{0.32}) has a
unique solution $(y,u)$ with $y,u \in W^1_{1,loc} (\mathbb{R})$
such that (\ref{0.33}) holds.

On the other hand, the coefficients of the system (\ref{0.32}) are
$\pi$--periodic, so one may apply the classical Floquet theory.

Let $(y_1,u_1) $ and $(y_2, u_2)$ be the solutions of
 (\ref{0.32}) which satisfy $y_1 (0)=1, u_1 (0) =0$ and
 $y_2(0)=0, u_2(0) =1.$  By the Caley--Hamilton theorem
 the Wronskian
 $$ det \begin{pmatrix} y_1 (x)  & y_2 (x) \\
 u_1 (x) & u_2 (x) \end{pmatrix}
\equiv 1 $$ because the trace of the coefficient matrix of the
system (\ref{0.32}) is zero.

If $(y(x),u(x)) $ is a solution of (\ref{0.32}) with initial data  $(a,b),$ then
$(y(x+\pi), u(x+\pi) ) $ is a solution also, correspondingly with
initial data
$$ \begin{pmatrix}  y(\pi)\\ u(\pi)  \end{pmatrix}
=  M \begin{pmatrix}  a\\ b  \end{pmatrix} , \quad
M=\begin{pmatrix}  y_1 (\pi) & y_2 (\pi) \\
 u_1 (\pi)  & u_2 (\pi)  \end{pmatrix}. $$
Consider the characteristic equation of the {\em monodromy matrix}
$M:$
\begin{equation}
\label{0.34}
\rho^2 - \Delta \rho + 1 = 0, \qquad \Delta =  y_1 (\pi) + u_2 (\pi) .
\end{equation}

Each root $\rho $ of the characteristic equation  (\ref{0.34})
gives a rise of a special solution $  (\varphi(x), \psi (x)) $
 of (\ref{0.32}) such that
 \begin{equation}
 \label{0.341}
 \varphi (x+\pi) = \rho \cdot \varphi (x), \quad \psi (x+\pi) = \rho
 \cdot \psi (x) .
 \end{equation}
Since the product of the
roots of (\ref{0.34}) equals $1,$
the roots have the form
\begin{equation}
\label{0.35} \rho^\pm = e^{\pm  \tau \pi}, \qquad \tau = \alpha+i
\beta,
\end{equation}
where $ \beta \in [0, 2] $ and $\alpha=0 $ if the roots are on the
unit circle or $ \alpha >0 $ otherwise.

In the case where the equation (\ref{0.34}) has two distinct
roots, let $  (\varphi^\pm, \psi^\pm)$ be special solutions of
(\ref{0.32}) that correspond to the roots (\ref{0.35}), i.e., $$
(\varphi^\pm (x+\pi), \psi^\pm (x+\pi) ) = \rho^\pm \cdot
(\varphi^\pm (x), \psi^\pm  (x) ). $$ Then one can readily see
that the functions $$ \tilde{\varphi}^\pm (x) = e^{\mp \tau x}
\varphi^\pm (x), \quad \tilde{\psi}^\pm (x) = e^{\mp \tau x}
\psi^\pm (x) $$ are $\pi$--periodic, and we have
\begin{equation}
\label{0.350}
 \varphi^\pm (x) =e^{\pm \tau x} \tilde{\varphi}^\pm (x),
\quad \psi^\pm (x) =e^{\pm \tau x} \tilde{\psi}^\pm (x).
\end{equation}

Consider the case where (\ref{0.34}) has a double root $\rho = \pm
1.$ If its geometric multiplicity equals 2
 (i.e., the matrix $M$ has two linearly independent eigenvectors),
then the equation (\ref{0.32}) has, respectively, two linearly
independent solutions $  (\varphi^\pm, \psi^\pm)$  which are
periodic if $\rho =1 $ or anti-periodic if $\rho = -1.$

Otherwise,  (if $M$ is a Jordan matrix), there are two linearly
independent vectors $\begin{pmatrix} a^+\\ b^+
\end{pmatrix} $ and $\begin{pmatrix} a^-\\ b^- \end{pmatrix} $
such that
\begin{equation}
\label{0.320} M
\begin{pmatrix} a^+\\ b^+ \end{pmatrix} = \rho
\begin{pmatrix} a^+\\ b^+ \end{pmatrix}, \quad
M
\begin{pmatrix} a^-\\ b^- \end{pmatrix} =\rho
\begin{pmatrix} a^-\\ b^- \end{pmatrix}
+ \rho \kappa  \begin{pmatrix} a^+\\ b^+ \end{pmatrix}, \quad
\rho=\pm 1, \; \kappa \neq 0.
\end{equation}
Let $  (\varphi^\pm, \psi^\pm)$
 be the corresponding solutions of (\ref{0.32}). Then we
have
\begin{equation}
\label{0.321}
\begin{pmatrix} \varphi^+ (x+\pi) \\ \psi^+ (x+\pi)  \end{pmatrix}
= \rho
\begin{pmatrix} \varphi^+ (x) \\ \psi^+ (x)  \end{pmatrix},
\qquad
\begin{pmatrix} \varphi^- (x+\pi) \\ \psi^- (x+\pi)  \end{pmatrix}
= \rho
\begin{pmatrix} \varphi^- (x) \\ \psi^- (x)  \end{pmatrix}
+ \rho \kappa \begin{pmatrix} \varphi^+ (x) \\ \psi^+ (x)
\end{pmatrix}.
\end{equation}
Now, one can easily see that the functions $\tilde{\varphi}^-$ and
$ \tilde{\psi}^-$ given by $$
\begin{pmatrix} \tilde{\varphi}^- (x) \\ \tilde{\psi}^- (x)
\end{pmatrix} =
\begin{pmatrix} \varphi^- (x) \\ \psi^- (x)  \end{pmatrix} -
\frac{\kappa x}{\pi} \begin{pmatrix} \varphi^+ (x) \\ \psi^+ (x)
\end{pmatrix} $$
are $\pi$--periodic (if $\rho =1$) or anti--periodic (if
$\rho=-1$). Therefore, the solution $\begin{pmatrix} \varphi^- (x)
\\ \psi^- (x)
\end{pmatrix} $ can be written in the form
\begin{equation}
\label{0.330}
\begin{pmatrix} \varphi^- (x) \\ \psi^- (x)
\end{pmatrix}=
\begin{pmatrix} \tilde{\varphi}^- (x) \\ \tilde{\psi}^-
(x)  \end{pmatrix} + \frac{ \kappa x}{\pi} \begin{pmatrix} \varphi^+ (x)
\\ \psi^+ (x)
\end{pmatrix},
\end{equation}
i.e., it is a linear combination of periodic (if $\rho =1$), or
anti--periodic (if $\rho = -1 $) functions with coefficients $1$
and $\kappa x /\pi.$

The following lemma shows how the properties of the solutions of
(\ref{0.31}) and (\ref{0.32}) depend on the roots of the
characteristic equation (\ref{0.34}).

\begin{Lemma}
\label{lem01} (a) The homogeneous system (\ref{0.32}) has no
nonzero solution $(y,u) $ with $ y \in L^2 (\mathbb{R}). $
Moreover, if the roots of the characteristic equation (\ref{0.34})
lie on the unit circle, i.e., $  \alpha =0 $ in the
representation (\ref{0.35}), then (\ref{0.32}) has no nonzero
solution $(y,u) $ with $ y \in L^2 ((-\infty,0] ) $ or $ y \in L^2
([0, +\infty ) ). $

 (b)  If
 $  \alpha=0 $ in the representation (\ref{0.35}), then there are functions
$f \in L^2 (\mathbb{R}) $ such that the corresponding
non-homogeneous system (\ref{0.31}) has no solution $(y, u)$ with
$ y \in L^2 (\mathbb{R}). $

 (c)   If  the roots of the characteristic equation (\ref{0.34})
lie outside the unit circle, i.e.,
 $  \alpha > 0 $ in the representation (\ref{0.35}), then
the non-homogeneous system (\ref{0.31}) has, for each $f \in L^2
(\mathbb{R}),$ a unique solution $ (y,u) = (R_1 (f), R_2 (f)) $
such that   $R_1 $ is a linear continuous operator  from $L^2
(\mathbb{R}) $ into $W^1_2 (\mathbb{R}), $  and $R_2 $ is a linear
continuous operator  in $L^2 (\mathbb{R}) $ with a range in
$W^1_{1,loc} (\mathbb{R}).$

\end{Lemma}

\begin{proof}

(a) In view of the above discussion
(see the text from (\ref{0.34}) to (\ref{0.330})), if the characteristic
equation (\ref{0.34}) has two  distinct roots $\rho = e^{\pm \tau \pi},$
then each solution $(y,u)$ of the homogeneous system (\ref{0.32})
is a linear combination of the special solutions, so  $$ y(x) =
C^+ e^{\tau x} \tilde{\varphi}^+ (x) + C^-  e^{-\tau x}
\tilde{\varphi}^- (x),$$
 where $\tilde{\varphi}^+ $ and $  \tilde{\varphi}^- $ are
$\pi$--periodic functions in $H^1.$

In the case where the real part of $\tau $ is strictly positive,
i.e., $ \tau = \alpha + i \beta $ with $\alpha >0,$ one can
readily see that $e^{\tau x} \tilde{\varphi}^+ (x) \not \in L^2
([0,\infty))$ but $e^{\tau x} \tilde{\varphi}^+ (x)  \in L^2
((-\infty,0]),$
 while $e^{-\tau x} \tilde{\varphi}^- (x) \in L^2
([0,\infty))$ but $e^{-\tau x} \tilde{\varphi}^- (x) \not \in L^2
((-\infty, 0])).$ Therefore,  if $y \not \equiv 0 $ we have $y
\not \in L^2 (\mathbb{R}).$

Next we consider the case where $\tau = i \beta $ with $\beta \neq
0, 1. $ The Fourier series of the functions $\tilde{\varphi}^+ (x)
$ and $\tilde{\varphi}^- (x)$ $$ \tilde{\varphi}^+ \sim \sum_{k
\in 2\mathbb{Z}} \tilde{\varphi}^+_k e^{ikx}, \quad
\tilde{\varphi}^- \sim \sum_{k \in 2\mathbb{Z}}
\tilde{\varphi}^-_k e^{ikx} $$ converge uniformly in $\mathbb{R} $
because $ \tilde{\varphi}^+, \tilde{\varphi}^- \in H^1. $
Therefore, we have $$y(x) =C^+ \sum_{k \in 2\mathbb{Z}}
\tilde{\varphi}^+_k e^{i(k+\beta)x}+ C^- \sum_{k \in 2\mathbb{Z}}
\tilde{\varphi}^-_k e^{i(k-\beta)x}, $$ where the series on the
right converge uniformly on $\mathbb{R}.$ If $ \beta $ is a
rational number, then $y $ is a periodic function, so $ y \not \in
L^2 ((-\infty,0])$ and $ y \not \in L^2 ([0,\infty)).$

If $\beta $ is an irrational number, then
\begin{equation}
\label{0.340} \lim_{T\to \infty} \frac{1}{T} \int_0^T y(x) e^{-
i(k \pm \beta)x}  dx = C^\pm \tilde{\varphi}^\pm_k \quad \forall k
\in 2\mathbb{Z}.
\end{equation}
On the other hand, if $y \in L^2 ([0,\infty)) $, then the Cauchy
inequality implies $$ \left | \frac{1}{T} \int_0^T y(x) e^{- i(k
\pm \beta)x} dx \right | \leq \frac{1}{T}  \left (\int_0^T 1\cdot
dx \right )^{1/2}  \left (\int_0^T |y(x)|^2 dx \right )^{1/2} \leq
\frac{\|y\|_{L^2 ([0,\infty))} }{\sqrt{T}} \to 0.$$ But, in view
of (\ref{0.340}), this is impossible if $y \neq 0.$ Thus $y \not
\in L^2 ([0,\infty)). $  In a similar way, one can see that $y
\not \in L^2 ((-\infty,0]) $.

Finally, if the characteristic equation (\ref{0.34}) has a double
root $\rho = \pm  1, $ then either every solution $(y,u)$ of
(\ref{0.32}) is periodic or anti--periodic, and so $y \not \in L^2
([0,\infty) $ and $y \not \in L^2 ((-\infty,0]), $ or it is a
linear combination of some special solutions (see (\ref{0.330}),
and the preceding discussion), so we have $$ y(x) = C^+ \varphi^+
(x) + C^- \tilde{\varphi}^- + C^- \frac{\kappa x}{\pi} \varphi^+
(x),$$ where the functions $\varphi^+ $ and $\tilde{\varphi}^- $
are periodic or anti--periodic. Now one can easily see that $y
\not \in L^2 ([0,\infty) $ and $y \not \in L^2 ((-\infty,0]), $
which completes the proof of (a). \vspace{2mm}

 (b) Let
$ (\varphi^\pm, \psi^\pm)$ be special solutions of (\ref{0.32})
that correspond to the roots (\ref{0.35}) as above. We may assume
without loss of generalities that the Wronskian  of the solutions
$(\varphi^+, \psi^+) $
 and $(\varphi^-, \psi^-) $ equals $1$ because these
 solutions are determined up to constant multipliers.

The standard method of variation of constants leads to the
following solution $(y,u)$ of the non--homogeneous system
(\ref{0.31}):
\begin{equation}
\label{0.351}
y = v^+ (x) \varphi^+ (x)  + v^- (x) \varphi^- (x),
\quad u= v^+ (x) \psi^+ (x)  + v^- (x) \psi^- (x),
\end{equation}
where $ v^+ $ and $v^- $ satisfy
\begin{equation}
\label{0.352}
 \frac{dv}{dx}^+ \cdot \varphi^+   + \frac{dv}{dx}^- \cdot
  \varphi^- = 0,
\quad \frac{dv}{dx}^+ \cdot  \psi^+   + \frac{dv}{dx}^- \cdot
\psi^- = f,
\end{equation}
so
\begin{equation}
\label{0.353} v^+ (x) = -\int_0^x \varphi^- (t) f(t) dt + C^+,
\quad v^- (x) = \int_0^x \varphi^+ (t) f(t) dt +C^-.
\end{equation}

Assume that the characteristic equation (\ref{0.34}) has roots of
the form $ \rho = e^{i \beta \pi}, \, \beta \in [0,2).$
 Take any function $f \in L^2 (\mathbb{R})$ with compact support,
say $\text{supp}\, f \subset (0,T).$  By (\ref{0.351}) and
(\ref{0.353}), if $(y,u)$ is a solution of the non-homogeneous
system (\ref{0.31}), then the restriction of $(y,u)$ on the
intervals $(-\infty, 0) $ and
 $[T, \infty )$ is a solution of the homogeneous system
(\ref{0.32}). So, by (a), if $y \in L^2 (\mathbb{R}) $ then
$y\equiv 0$ on the intervals $(-\infty, 0) $ and
 $[T, \infty ). $ This may happen if only if
 the constants $C^\pm $ in (\ref{0.353}) are zeros, and we have
 $$ \int_0^T \varphi^- (t) f(t) dt = 0, \quad
\int_0^T \varphi^+ (t) f(t) dt = 0. $$

Hence, if $f$ is not orthogonal to the functions $\varphi^\pm, $
on the interval $[0,T],$ then the non--homogeneous system
(\ref{0.31}) has no solution $(y,u)$ with $y\in L^2(\mathbb{R}).$
This completes the proof of (b). \vspace{2mm}

(c) Now we consider the case where the characteristic equation
(\ref{0.34})  has roots of the form (\ref{0.35}) with $\alpha >0.$
Let $(\varphi^\pm, \psi^\pm) $ be the corresponding special
solutions. By (\ref{0.351}), for each $f \in L^2 (\mathbb{R}),$
the non-homogeneous system (\ref{0.31}) has a solution of the form
$(y,u)=(R_1 (f), R_2 (f),$ where
\begin{equation}
\label{0.344} R_1 (f) = v^+ (x) \varphi^+ (x)  + v^- (x) \varphi^-
(x), \quad R_2 (f)= v^+ (x) \psi^+ (x)  + v^- (x) \psi^- (x),
\end{equation}
and (\ref{0.352}) holds.  In order to have a solution
 that vanishes at $\pm \infty $ we set (taking into account
 (\ref{0.350}))
\begin{equation}
\label{0.354} v^+ (x) =  \int_x^\infty e^{-\tau t}
\tilde{\varphi}^- (t) f(t) dt, \quad v^- (x) =  \int_{-\infty}^x
e^{\tau t} \tilde{\varphi}^+ (t) f(t) dt.
\end{equation}

Let  $ C_\pm = \max \{ |\tilde{\varphi}^\pm (x)|: \; x \in
[0,\pi]\}.$ By (\ref{0.350}), we have
\begin{equation}
\label{0.355}
|\varphi^\pm (x)|   \leq C_\pm  \cdot e^{\pm \alpha x}.
\end{equation}
Therefore, by the Cauchy inequality,  we get $$ |v^+ (x) |^2 \leq
C^2_- \left | \int_x^\infty e^{-\alpha t } |f(t)| dt \right |^2
\leq C^2_-  \left  (    \int_x^\infty e^{-\alpha t} dt \right )
\cdot
 \left  (    \int_x^\infty e^{-\alpha t} |f(t)|^2 dt \right ),
 $$
 so
\begin{equation}
\label{0.356}
  |v^+ (x) |^2
 \leq
\frac{C^2_-}{\alpha  }
e^{-\alpha x} \int_x^\infty e^{-\alpha t} |f(t)|^2 dt.
\end{equation}
 Thus, by (\ref{0.355}),
 $$
 \int_{-\infty}^\infty \left | v^+(x) \right |^2 \left | \varphi^+(x) \right |^2 dx
 \leq   \frac{C^2_- C_+^2}{\alpha }
 \int_{-\infty}^\infty e^{\alpha x} \int_x^\infty e^{-\alpha t} |f(t)|^2 dt dx $$
 $$ \leq  \frac{C^2_- C_+^2}{\alpha }
 \int_{-\infty}^\infty |f(t)|^2  \left ( \int_{-\infty}^t e^{\alpha (x-t)} dx
 \right ) dt = \frac{C^2_- C_+^2}{\alpha^2 } \|f\|^2_{L^2(\mathbb{R})}. $$

In an analogous way one may prove that
$$
 \int_{-\infty}^\infty \left | v^-(x) \right |^2 \left | \varphi^-(x) \right |^2 dx
 \leq \frac{C^2_- C_+^2}{\alpha^2 } \|f\|^2_{L^2(\mathbb{R})}. $$
In view of (\ref{0.351}), these estimates prove that $ R_1 $ is a continuous
operator in $L^2 (\mathbb{R}). $

Next we estimate the $L^2 (\mathbb{R})$--norm of $y^\prime =
\frac{d}{dx} R_1 (f). $ In view of (\ref{0.352}), we have
$$y^\prime (x)  =  v^+ (x) \cdot  \frac{d\varphi}{dx}^+ (x)  + v^-
(x) \cdot \frac{d\varphi}{dx}^- (x).  $$

By (\ref{0.350}), $$ v^+ (x) \cdot  \frac{d\varphi}{dx}^+ (x) =
\alpha v^+ (x)  \varphi^+  + v^+ (x)  e^{\alpha x}
\frac{d\tilde{\varphi}}{dx}^+. $$ Since the $L^2
(\mathbb{R})$--norm of $ v^+ (x)  \varphi^+ $ has been estimated
above, we need to estimate only the $L^2 (\mathbb{R})$--norm of
$v^+ (x)  e^{\alpha x}  d\tilde{\varphi}^+/dx. $ By (\ref{0.356}),
we have $$ \int_{-\infty}^\infty \left | v^+ (x)  e^{\alpha x}
d\tilde{\varphi}^+/dx \right |^2  dx \leq \frac{C_-^2}{\alpha}
\int_{-\infty}^{\infty} \left | d\tilde{\varphi}^+/dx \right |^2
e^{\alpha x}  \int_x^\infty e^{-\alpha t} |f(t)|^2 dt dx $$ $$=
\frac{C_-^2}{\alpha} \int_{-\infty}^{\infty} |f(t)|^2   \left (
\int_{-\infty}^t \left | d\tilde{\varphi}^+/dx \right |^2
e^{\alpha (x-t)}  dx  \right ) dt. $$ Firstly, we estimate the
integral in the parentheses. Notice that the function
$d\varphi^\pm /dx $ (and therefore, $ d\tilde{\varphi}^\pm/dx $ )
are in the space $L^2 ([0,\pi])$ due to the first equation in
(\ref{0.32}). Therefore,
\begin{equation}
\label{0.358} K_\pm^2 = \int_0^\pi \left |
\frac{d\tilde{\varphi}}{dx}^\pm (x) \right |^2 dx  < \infty.
\end{equation}
We have
 $$ \int_{-\infty}^t
\left | d\tilde{\varphi}^+/dx \right |^2 e^{\alpha (x-t)}  dx  =
\sum_{n=0}^\infty \int_{-(n+1)\pi}^{-n \pi} \left |
\frac{d\tilde{\varphi}}{dx}^+ (\xi+t) \right |^2  e^{\alpha \xi}
d\xi $$ $$ \leq K^2_+ \cdot \sum_{n=0}^\infty e^{-\alpha n \pi } =
\frac{K^2_+}{1- \exp (-\alpha \pi )} < (1+\alpha \pi)
\frac{K^2_+}{\alpha \pi}. $$ Thus, $$ \int_{-\infty}^\infty \left
| v^+ (x)  e^{\alpha x} d\tilde{\varphi}^+/dx \right |^2  dx \leq
(1+\alpha \pi) \frac{C_-^2 K^2_+}{\alpha^2 \pi} \|f\|^2 . $$

In an analogous way it follows that $$ \int_{-\infty}^\infty \left
| v^- (x)  e^{\alpha x}  \frac{d\tilde{\varphi}}{dx}^- \right |^2
dx \leq (1+\alpha \pi) \frac{C_+^2 K^2_-}{\alpha^2 \pi} \|f\|^2 ,
$$ so the operator $R_1 $ act continuously from $L^2 (\mathbb{R})
$ into the space $W^1_2 (\mathbb{R}).$

The proof of the fact that
 the operator $R_2 $ is continuous in $L^2 (\mathbb{R}) $
 is omitted because  essentially it is
 the same (we only replace $\varphi^\pm $ with
 $\psi^\pm $ in the proof that $R_1$ is a
 continuous operator in $L^2 (\mathbb{R})$).
  \end{proof}

We need also the following lemma.

\begin{Lemma}
\label{lem001} Let $H$ be a Hilbert space with product
$(\cdot,\cdot),$ and let $$ A: \; D(A) \to H, \quad B: \; D(B) \to
H $$ be (unbounded) linear operators with domains $D(A)$ and
$D(B), $ such that
\begin{equation}
\label{0.390} (A f,g) = (f,B g) \quad \text{for} \;\; f\in D(A),
\; g\in D(B).
\end{equation}
If there is a $\lambda \in \mathbb{C} $ such that the operators
$A-\lambda $ and $B-\overline{\lambda} $ are surjective, then

(i)   $ D(A) $  and $D(B) $ are dense in $H;$

(ii)  $A^* = B $ and $  B^* = A, $ where $ A^* $ and $B^*$ are,
respectively, the adjoint operators of $A$ and $B.$
\end{Lemma}

\begin{proof}
We need to explain only that $D(A)$ is dense in $H$ and $A^* = B $
because one can replace the roles of $A$ and $ B.$

To prove that $D(A) $ is dense in $H,$ we need to show that if $h$
is orthogonal to $D(A) $ then $h=0.$ Let $$ (f, h) = 0 \quad
\forall f \in D(A). $$ Since the operator $B- \overline {\lambda}
$ is surjective, there is $g \in D(B) $ such that $h = (B-
\overline {\lambda})g. $ Therefore, by (\ref{0.390}), we have $$
0= (f, h) = (f,(B- \overline {\lambda})g) = ((A-\lambda)f, g)
\quad \forall f \in D(A), $$ which yields $ g=0 $ because the
range of $A-\lambda $ is $H.$ Thus, $ h =(B- \overline {\lambda})g
=0. $ Hence (i) holds.

Next we prove (ii). If $g^* \in Dom (A^*), $ then we have

\begin{equation}
\label{0.38}
(A-\lambda)f, g^*) = (f,w) \quad \forall f \in D(A),
\end{equation}
 where
$w=(A^* - \overline{\lambda})g^*.$ Since the operator $B-
\overline {\lambda} $ is surjective, there is $g \in D(B) $ such
that $w = (B- \overline {\lambda})g. $ Therefore, by (\ref{0.390})
and (\ref{0.38}), we have $$ ((A-\lambda)f, g^*) = (f,(B-
\overline {\lambda})g) =((A-\lambda)f, g) \quad \forall f \in
D(A),$$ which implies that $g^* = g$ (because the range of $A
-\lambda $ is equal to $H$) and $(A^* - \overline{\lambda})g^* =
(B - \overline{\lambda})g^*, $ i.e., $A^* g^* = B g^*. $ This
completes the proof of (ii).

\end{proof}

Consider the Schr\"odinger operator with a spectral parameter $$
L(v) -\lambda = -d^2/dx^2 + (v- \lambda), \quad \lambda \in
\mathbb{C}.$$ In view of the formula (\ref{0.13} in Proposition
\ref{prop01}, we may assume without loss of generality that
\begin{equation}
\label{0.39}
 C = 0, \quad v= Q^\prime,
\end{equation}
because a change of $C $ results in a shift of the spectral
parameter $\lambda.$

Replacing  $C$ by $ - \lambda $
in the homogeneous system (\ref{0.32}),
we get
\begin{equation}
\label{0.44}  \left | \begin{array}{l}  y^\prime = Qy +u,\\
u^\prime = (-\lambda - Q^2) y - Qu. \end{array}  \right.
\end{equation}
Let $(y_1 (x;\lambda),u_1(x;\lambda)) $  and $(y_2(x;\lambda),
u_2(x;\lambda))$ be the solutions of
 (\ref{0.44}) which satisfy the initial conditions
 $y_1 (0;\lambda)=1, u_1 (0;\lambda) =0  $ and
 $y_2(0;\lambda)=0, u_2(0;\lambda) =1.$
 Since these solutions depend analytically on $\lambda \in \mathbb{C},
 $ the {\em Lyapunov function}, or {\em Hill discriminant},
\begin{equation}
\label{0.45} \Delta (Q,\lambda ) =  y_1 (\pi;\lambda) +u_2
(\pi;\lambda)
\end{equation}
is an entire function. Taking the conjugates of the equation in
(\ref{0.44}), one can easily see that
\begin{equation}
\label{0.46} \Delta (\overline{Q},\overline{\lambda} ) =
\overline{\Delta (Q, \lambda)}.
\end{equation}

{\em Remark.} A. Savchuk and  A. Shkalikov  gave asymptotic
analysis of the functions  $y_j (\pi, \lambda) $ and $u_j (\pi,
\lambda), \, j=1,2.$ In particular, it follows from Formula (1.5)
of Lemma~1.4  in \cite{SS03} that, with $ z^2 = \lambda, $
\begin{equation}
\label{00.45}    y_1 (\pi, \lambda) =    \cos (\pi z)  + o(1),
\quad y_2 (\pi, \lambda) = \frac{1}{z} [  \sin (\pi z)  + o(1)],
\quad u_2 (\pi, \lambda) = \cos \pi z  + o(1),
\end{equation}
and therefore,
\begin{equation}
\label{00.46}
\Delta (Q, \lambda) = 2 \cos \pi z  + o(1), \quad z^2 = \lambda,
\end{equation}
inside any parabola
\begin{equation}
\label{00.47}
P_a = \{\lambda \in \mathbb{C}: \quad |Im \, z| \leq a \}.
\end{equation}
In the regular case $v \in L^2([0,\pi])$ these asymptotics of the
fundamental solutions and the Lyapunov function $\Delta $  of the
Hill--Schr\"odinger operator could be found in \cite{Mar}, p. 32,
Formula (1.3.11), or pp. 252-253, Formulae ($3.4.23^\prime$),
(3.4.26).\vspace{3mm}

Consider the operator $L(v), $ in the domain
\begin{equation}
\label{0.47} D(L(v)) =
 \left  \{ y \in H^1 (\mathbb{R}):
 \;\; y^\prime - Qy \in L^2 (\mathbb{R}) \cap W^1_{1,loc} (\mathbb{R}), \;\;
 \ell_Q (y) \in L^2 (\mathbb{R})   \right \},
\end{equation}
defined by
\begin{equation}
\label{0.48} L(v) y = \ell_Q (y), \quad \text{with}  \;\; \ell_Q
(y)
 = - (y^\prime -Qy)^\prime - Qy^\prime,
\end{equation}
where $v$ and $Q$ are as in Proposition~\ref{prop01}.

\begin{Theorem}
\label{thm1} Let   $ v \in
H^{-1}_{loc} (\mathbb{R}))$ be $\pi$--periodic.
Then

(a) the domain $D(L(v)) $ is dense in $L^2
(\mathbb{R}); $

(b) the operator $L(v)$ is closed, and its conjugate operator is
\begin{equation}
\label{0.50} (L(v))^* = L(\overline{v});
\end{equation}
(In particular,  if $v$ is real--valued, then the operator $L(v)$
is self--adjoint.)

(c) the spectrum $Sp (L(v))$ of the operator $L(v) $ is
continuous, and moreover,
\begin{equation}
\label{0.51} Sp (L(v)) = \{\lambda \in \mathbb{C} \; |  \quad
\exists \theta  \in [0,2\pi): \;\; \Delta (\lambda) = 2 \cos
\theta \}.
\end{equation}
\end{Theorem}

{\em Remark.} In the case of $L^2 $--potential $v$ this result is
known (see Rofe--Beketov \cite{R-B,R-BK} and V. Tkachenko
\cite{Tk64}).

\begin{proof}
Firstly, we show that the operators $L(v) $ and $L(\overline{v}) $
are formally adjoint, i.e.,
\begin{equation}
\label{0.52} \left (L(v)y,h \right ) = \left (f, L(\overline{v})
 h \right ) \quad \text{if} \;\; y \in
D(L(v)), \;\; h \in D(L(\overline{v})).
\end{equation}
Since $y^\prime - Q y$ and $\overline{h}$ are continuous $L^2
(\mathbb{R})$--functions, their product is a continuous $L^1
(\mathbb{R})$--function, so
 we have $$ \liminf_{x\to \pm \infty}
\left | (y^\prime - Q y) \overline{h} \right | (x) = 0.$$
Therefore, there exist two sequences of real numbers $c_n \to
-\infty $ and $d_n \to \infty $ such that $$ \left ((y^\prime - Q
y) \overline{h} \right ) (c_n) \to 0, \quad \left ((y^\prime - Q
y) \overline{h} \right ) (d_n) \to  0 \quad \text{as} \;\; n \to
\infty.$$ Now, we have  $$ \left (L(v) y, h \right ) =
\int_{-\infty}^\infty  \ell_Q (y) \overline{h} dx = \lim_{n \to
\infty} \int_{c_n}^{d_n} \left (-(y^\prime -Qy)^\prime
\overline{h} - Qy^\prime \overline{h}\right ) dx $$ $$ = \lim_{n
\to \infty} \left ( -(y^\prime -Qy)\overline{h}\text{\huge
$\vert$}_{c_n}^{d_n} + \int_{c_n}^{d_n} (y^\prime -Qy)
\overline{h^\prime} dx - \int_{c_n}^{d_n} Qy^\prime \overline{h}
dx \right )  $$ $$ = 0 + \int_{-\infty}^\infty \left ( y^\prime
\overline{h^\prime} - Q y \overline{h^\prime} - Q y^\prime
\overline{h} \right ) dx. $$

The same argument shows that $$ \int_{-\infty}^\infty \left (
y^\prime \overline{h^\prime} - Q y \overline{h^\prime} - Q
y^\prime \overline{h} \right ) dx = \left (y, L(\overline{v}) h
\right ), $$ which completes the proof of (\ref{0.52}).

If the roots of the characteristic equation $\rho^2 - \Delta
(Q,\lambda ) \rho + 1 = 0$ lie on the unit circle $
\{e^{i\theta}, \theta \in [0,2\pi)\},$ then they are of the form
$e^{\pm i\theta},$ so we have
\begin{equation}
\label{0.54} \Delta (Q,\lambda) = e^{i\theta} + e^{-i\theta} = 2
\cos \theta.
\end{equation}
Therefore, if $\Delta (Q,\lambda ) \not \in [-2,2],$ then the
roots of the characteristic equation
 lie outside of the unit circle  $ \{e^{i\theta},
\theta \in [0,2\pi)\}.$ If so, by part (c) of Lemma~\ref{lem01},
the operator $L(v) - \lambda$
 maps bijectively $D(L(v))$ onto $L^2(\mathbb{R}),$
and its inverse operator $$ (L(v))-\lambda)^{-1}: \;
L^2(\mathbb{R}) \to D(L(v)) $$ is a continuous linear operator.
Thus,
\begin{equation}
\label{0.56}
 \Delta (Q,\lambda)  \not \in [-2,2] \Rightarrow
 (L(v)-\lambda)^{-1}: \;
L^2(\mathbb{R}) \to D(L(v)) \quad {exists}.
  \end{equation}

Next we apply Lemma \ref{lem001} with $A= L(v)$ and
$B=L(\overline{v}).$  Choose $\lambda \in \mathbb{C} $ so that $
\Delta (Q, \lambda) \not \in [-2,2] $ (in view of (\ref{00.46}),
see the remark before Theorem~\ref{thm1}, $ \Delta (Q, \lambda) $
is a non--constant entire function, so such a choice is possible).
Then, in view of (\ref{0.46}), we have that $ \Delta
(\overline{Q}, \overline{\lambda}) \not \in [-2,2] $ also. In view
of the above discussion, this means that the operator $ L(v) -
\lambda $ maps bijectively $D(L(v)) $ onto $L^2(\mathbb{R}) $ and
$L(\overline{v}) - \overline{\lambda} $ maps bijectively
$D(L(\overline{v})) $ onto $L^2(\mathbb{R}). $ Thus, by
Lemma~\ref{lem001}, $D(L(v))$ is dense in $L^2 (\mathbb{R})$ and
$L(v)^* = L(\overline{v}), $ i.e., (a) and (b) hold.

 Finally, in view of (\ref{0.56}), (c) follows readily from
 part (b) of Lemma \ref{lem01}.

\end{proof}

3. Theorem \ref{thm1} shows that the spectrum of the operator
$L(v) $ is described by the equation ({\ref{0.51}). As we are
going to explain below, this fact implies that the spectrum $Sp
(L(v))$ could be described in terms of the spectra of the
operators $L_\theta = L_\theta (v), \; \theta \in [0,\pi ], $ that
arise from the same differential expression $\ell = \ell_Q  $ when
it is considered on the interval $[0,\pi]$ with the following
boundary conditions:
\begin{equation}
\label{0.70}
y(\pi) = e^{i\theta} y(0), \quad (y^\prime - Qy)(\pi) = e^{i\theta}
(y^\prime - Qy)(0).
\end{equation}
The domains $D(L_\theta) $ of the operators
$L_\theta $ are given by
\begin{equation}
\label{0.71} D(L_\theta) =\left  \{y \in H^1 \; :  \;\; y^\prime -
Qy \in  W^1_1 ([0,\pi]),
 \;\; \text{(\ref{0.70})} \; holds,
 \;\; \ell (y) \in H^0  \right \},
\end{equation}
where $$ H^1 = H^1 ([0,\pi]), \quad H^0 = L^2 ([0,\pi]). $$

We set
\begin{equation}
\label{0.72}
L_\theta  ( y )= \ell (y),  \quad  y \in D(L_\theta).
\end{equation}
Notice that if $y \in H^1 ([0,\pi]), $ then $ \ell_Q (y) = f \in
L^2 ([0,\pi]) $ if and only if $u = y^\prime - Q y \in W^1_1
([0,\pi]) $ and the pair $(y,u) $ is a solution of the
non--homogeneous system (\ref{0.31}).

\begin{Lemma}
\label{lem02} Let $ \begin{pmatrix} y_1 \\ u_1     \end{pmatrix} $
and $ \begin{pmatrix} y_2 \\ u_2     \end{pmatrix} $ be the
solutions of the homogeneous system (\ref{0.32}) which satisfy
\begin{equation}
\label{0.70a}
\begin{pmatrix} y_1 (0) \\ u_1 (0)     \end{pmatrix}
=  \begin{pmatrix} 1 \\ 0     \end{pmatrix}, \qquad
\begin{pmatrix} y_2 (0) \\ u_2 (0)     \end{pmatrix}
=  \begin{pmatrix} 0 \\ 1     \end{pmatrix}.
\end{equation}
 If
\begin{equation}
\label{0.73} \Delta = y_1 (\pi) + u_2 (\pi) \neq 2 \cos \theta,
\qquad  \theta \in [0,\pi],
\end{equation}
then the non--homogeneous system (\ref{0.31}) has, for each $f \in
H^0,$ a unique solution $(y,u) = (R_1 (f), R_2 (f))$ such that
\begin{equation}
\label{0.74}
\begin{pmatrix}
y(\pi)\\u(\pi)
\end{pmatrix}
= e^{i\theta}
\begin{pmatrix}
y(0)\\u(0)
\end{pmatrix}.
\end{equation}
Moreover, $R_1  $ is a linear continuous operator from $H^0$ into
$H^1, $ and $R_2 $ is a linear continuous operator in $H^0 $ with
a range in $W^1_1 ([0,\pi]).$
\end{Lemma}

\begin{proof}
By the variation of parameters method, every solution of the
non--homogeneous system (\ref{0.31}) has the form
\begin{equation}
\label{0.75}
\begin{pmatrix} y (x) \\ u (x)     \end{pmatrix}
= v_1 (x)
\begin{pmatrix} y_1 (x) \\ u_1 (x)     \end{pmatrix}
+ v_2 (x) \begin{pmatrix} y_2 (x) \\ u_2 (x)     \end{pmatrix},
\end{equation}
where
\begin{equation}
\label{0.77} v_1 (x) = -\int_0^x  y_2 (x) f(t)dt +  C_1,     \quad
v_2 (x) = \int_0^x  y_1 (x) f(t) dt +  C_2.
\end{equation}
We set for convenience
\begin{equation}
\label{0.79} m_1 (f) = -\int_0^\pi y_2 (t) f(t) dt, \quad m_2 (f)
= \int_0^\pi y_1 (t) f(t) dt.
\end{equation}
By (\ref{0.75})--(\ref{0.79}), the condition (\ref{0.74}) is
equivalent to
\begin{equation}
\label{0.80} \left (m_1 (f) + C_1 \right ) \begin{pmatrix} y_1
(\pi)
\\ u_1 (\pi) \end{pmatrix}+ \left (m_2 (f) + C_2 \right ) \begin{pmatrix} y_2
(\pi) \\ u_2 (\pi) \end{pmatrix} = e^{i\theta}  \begin{pmatrix}
C_1 \\ C_2
\end{pmatrix}.
\end{equation}
This is a system of two linear equations in two unknowns $C_1 $
and $C_2.$ The corresponding determinant is equal to

$$ \det \begin{pmatrix} y_1 (\pi) - e^{i\theta} & y_2 (\pi) \\ u_1
(\pi) &   u_2 (\pi) - e^{i\theta }
\end{pmatrix}
=
1  + e^{2i\theta} - \Delta \cdot e^{i\theta} = e^{i \theta} (2
\cos \theta - \Delta). $$ Therefore, if (\ref{0.73}) holds, then
the system (\ref{0.80}) has a unique solution  $\begin{pmatrix}
C_1 \\ C_2  \end{pmatrix}, $ where $C_1 = C_1 (f) $ and $C_2 = C_2
(f) $ are linear combinations of $ m_1 (f) $ and $m_2 (f).$
 With
these values of $C_1 (f) $ and $C_2 (f) $ we set $$ R_1 (f) = v_1
\cdot y_1 + v_2 \cdot y_2 , \quad R_2 (f) = v_1  \cdot u_1 + v_2
\cdot u_2. $$ By (\ref{0.77}) and  (\ref{0.79}), the Cauchy
inequality implies $$ |v_1 (x)| \leq  \int_0^x | y_2 (t) f(t)| dt
+ \left |C_1 (f) \right | \leq  A \cdot \|f\|, \quad |v_2 (x)|
\leq B \cdot \|f\|,$$ where $A$ and $B$ are constants.
 From here it follows that  $R_1 $ and $R_2 $ are continuous linear
operators in $H^0.$ Since $$ \frac{d}{dx} R_1 (f) = v_1
\frac{d y_1}{dx} + v_2 \frac{dy_2}{dx}, \quad R_2 (f) = v_1
\frac{du_1}{dx} + v_2 \frac{d u_2}{dx} + f, $$ it follows also that
$R_1 $ acts continuously from $H^0 $ into $H^1, $ and $R_2 $ has
range in $W^1_1 ([0,\pi]),$ which completes the proof.

\end{proof}

\begin{Theorem}
\label{thm2}
Suppose $ v \in H^{-1}_{loc} (\mathbb{R}))$ is
$\pi$--periodic.
Then,

(a) for each $\theta \in [0, \pi] , $ the domain $D(L_\theta (v) )
\in (\ref{0.71}) $ is dense in $H^0; $

(b) the operator $L_\theta (v)  \in (\ref{0.72}) $
is closed, and its conjugate
operator is
\begin{equation}
\label{0.104}
L_\theta (v) ^*  =
L_\theta (\overline{v}).
\end{equation}
In particular,  if $v$ is real--valued, then the operator
$L_\theta (v)$ is self--adjoint.

(c) the spectrum $Sp (L_\theta (v) )$ of
the operator $L_\theta (v)  $ is
discrete, and moreover,
\begin{equation}
\label{0.105} Sp (L_\theta (v) ) = \{\lambda \in \mathbb{C} \; :
\;\; \Delta (\lambda) = 2 \cos \theta \}.
\end{equation}
\end{Theorem}

\begin{proof}
Firstly, we show that the operators $L_\theta (v)  $ and $L_\theta (\overline{v}) $
are formally adjoint, i.e.,
\begin{equation}
\label{0.107} \left (L_\theta (v)y,h \right ) =
\left (f, L_\theta(\overline{v})
 h \right ) \quad \text{if} \;\; y \in
D(L_\theta (v)), \;\; h \in D(L_\theta(\overline{v})).
\end{equation}

Indeed, in view of (\ref{0.70}),
 we have  $$ \left (L_\theta(v) y, h
\right ) =\frac{1}{\pi} \int_0^\pi  \ell_Q (y) \overline{h} dx =\frac{1}{\pi}
\int_0^\pi\left (-(y^\prime -Qy)^\prime
\overline{h} - Qy^\prime \overline{h}\right ) dx $$
$$ =  - \frac{1}{\pi} (y^\prime -Qy)\overline{h}\text{\huge
$\vert$}_{0}^{\pi} +
\frac{1}{\pi} \int_{0}^{\pi} (y^\prime -Qy)
\overline{h^\prime} dx - \int_{0}^{\pi} Qy^\prime \overline{h}
dx   $$ $$ = 0 +
\frac{1}{\pi} \int_0^\pi  \left ( y^\prime
\overline{h^\prime} - Q y \overline{h^\prime} - Q y^\prime
\overline{h} \right ) dx. $$

The same argument shows that $$ \frac{1}{\pi} \int_0^\pi \left (
y^\prime \overline{h^\prime} - Q y \overline{h^\prime} - Q
y^\prime \overline{h} \right ) dx = \left (y, L_\theta
(\overline{v}) h \right ), $$ which completes the proof of
(\ref{0.107}).

Now we apply Lemma \ref{lem001} with $A= L_\theta (v) $ and
$B=L_\theta (\overline{v}).$  Choose $\lambda \in \mathbb{C} $ so
that $ \Delta (Q, \lambda) \neq 2 \cos \theta $ (as one can easily
see from the remark before Theorem~\ref{thm1}, $ \Delta (Q,
\lambda) $ is a non--constant entire function, so such a choice is
possible). Then, in view of (\ref{0.46}), we have that $ \Delta
(\overline{Q}, \overline{\lambda}) \neq 2 \cos \theta $ also. By
Lemma~\ref{lem02},
 $ L_\theta (v) - \lambda $ maps bijectively $D(L_\theta(v)) $ onto
$H^0$ and $L_\theta(\overline{v}) - \overline{\lambda} $ maps
bijectively $D(L_\theta(\overline{v})) $ onto $H^0. $ Thus, by
Lemma~\ref{lem001}, $D(L_\theta (v) $ is dense in $H^0$ and
$L_\theta (v) ^* = L_\theta (\overline{v}), $ i.e., (a) and (b)
hold.

If $\Delta (Q, \lambda ) = 2 \cos \theta, $ then $e^{i \theta} $
is a root of the characteristic equation (\ref{0.34}), so there is
a special solution $(\varphi, \psi)$ of the homogeneous system
(\ref{0.32}) (considered with $C= - \lambda$) such that
(\ref{0.341}) holds with $\rho =e^{i\theta}.$ But then $ \varphi
\in D(L_\theta (v) ) $ and $L_\theta (v)  \varphi = \lambda
\varphi, $ i.e., $\lambda $ is an eigenvalue of $L_\theta (v) . $
In view of Lemma~\ref{lem02}, this means that (\ref{0.105}) holds.
Since $\Delta (Q, \lambda )$ is a non--constant entire function
(as one can easily see from the remark before
Theorem~\ref{thm001}) the set on the right in (\ref{0.105}) is
discrete. This completes the proof of (c).
\end{proof}

\begin{Corollary}
In view of Theorem \ref{thm1} and Theorem \ref{thm2},
we have
\begin{equation}
\label{cor} Sp \, (L(v)) = \bigcup_{\theta \in [0,\pi]} Sp \,
(L_\theta (v)).
\end{equation}
\end{Corollary}

In the self--adjoint case  (i.e., when
$v,$ and therefore,  $Q$ are
real--valued) the spectrum $  Sp \, (L(v)) \subset \mathbb{R}$
has a band--gap structure.  This is a well--known
 result in the regular case where
$v $ is an $L^2_{loc} (\mathbb{R}) $--function.
Its generalization in the singular case
 was proved by R. Hryniv and
Ya. Mykytiuk \cite{HM01}.

In order to formulate that result more precisely, let us
consider the following boundary conditions (bc):

($\text{a}^*$) periodic $\quad Per^+: \quad y(\pi)= y(0), \;\left ( y^\prime
- Qy  \right ) (\pi)= \left ( y^\prime - Qy  \right ) (0);   $

($\text{b}^*$) antiperiodic $\quad  Per^- : \quad y(\pi)= - y(0), \;\left (
y^\prime - Qy  \right )(\pi)= - \left ( y^\prime - Qy  \right )
(0); $

Of course, in the case where $Q$ is a continuous function, $Per^+
$ and $Per^-$  coincide, respectively, with the classical periodic
boundary condition $ y(\pi) = y(0), \; y^\prime (\pi) = y^\prime
(0) $ or anti--periodic boundary condition $ y(\pi) = - y(0), \;
y^\prime (\pi) = - y^\prime (0) $ (see the related discussion in
Section 6.2).

The boundary conditions $Per^\pm $ are particular cases of
(\ref{0.74}), considered, respectively, for $\theta =0$ or $\theta
= \pi. $ Therefore, by Theorem~\ref{thm2}, for each of these two
boundary conditions, the differential expression ({\ref{001})
gives a rise of a closed (self adjoint for real $v$) operator
$L_{Per^\pm} $ in $H^0 = L^2 ([0,\pi]), $ respectively, with a
domain
\begin{equation}
\label{002} D (L_{Per^+}) =\{y \in H^1: \;\;y^\prime - Qy \in
W^1_1 ([0,\pi]), \;\; (\text{a}^*) \; holds, \; \; l(y) \in H^0
\},
\end{equation}
or
\begin{equation}
\label{003} D (L_{Per^-}) =\{y \in H^1: \;\; y^\prime - Qy \in
W^1_1 ([0,\pi]), \;\; (\text{b}^*) \; holds, \; \; l(y) \in H^0
\}.
\end{equation}

The spectra of the operators $L_{Per^\pm} $
are discrete.  Let us enlist their eigenvalues
in increasing order, by
using
even indices for the eigenvalues of $L_{Per^+} $
and odd indices for the eigenvalues of
$L_{Per^-} $  (the convenience of such enumeration
will be clear later):

\begin{equation}
\label{102} Sp \, (L_{Per^+}) =\{\lambda_0, \lambda_2^-,
\lambda_2^+, \lambda_4^-,
\lambda_4^+,  \lambda_6^-,
\lambda_6^+, \ldots
\},
\end{equation}
\begin{equation}
\label{103} Sp \, (L_{Per^-}) =\{
\lambda_1^-, \lambda_1^+, \lambda_3^-,
\lambda_3^+, \lambda_5^-, \lambda_5^+ \ldots
\}.
\end{equation}

\begin{Proposition}
\label{prop03} Suppose $v = C+ Q^\prime,$ where $ Q \in L^2_{loc}
(\mathbb{R}))$ is a $\pi$--periodic real valued function.
 Then, in the above notations, we have
\begin{equation}
\label{104} \lambda_0  < \lambda_1^- \leq \lambda_1^+ <
\lambda_2^- \leq \lambda_2^+ <
\lambda_3^- \leq \lambda_3^+ <
\lambda_4^- \leq \lambda_4^+ <
\lambda_5^- \leq \lambda_5^+ < \cdots.
\end{equation}
Moreover, the spectrum of the operator $L (v) $ is absolutely
continuous and has a band--gap structure: it is a union of closed
intervals separated by spectral gaps $$ (-\infty, \lambda_0), \;
(\lambda^-_1,\lambda^+_1), \; (\lambda^-_2,\lambda^+_2),  \cdots,
(\lambda^-_n,\lambda^+_n), \cdots. $$
\end{Proposition}

Let us mention that
 A. Savchuk and A. Shkalikov \cite{SS00}
have studied the Sturm--Liouville operators that arise when the
differential expression $\ell_Q, \; Q \in L^2 ([0,1]), $ is
considered with appropriate regular boundary conditions (see
Theorems 1.5 and 1.6 in \cite{SS03}).

\section{Fourier representation of the operators $L_{Per^\pm} $}

Let $L^0_{bc}$ denote the free operator $L^0 = - d^2/dx^2 $
considered with boundary conditions $bc$ as a self--adjoint
operator in $L^2 ([0,\pi]).$ It is easy to describe the spectra and
eigenfunctions of $L^0_{bc}$ for $bc = Per^\pm, Dir :$

(a) $ Sp (L^0_{Per^+}) = \{n^2, \; n = 0,2,4, \ldots \};$ its
eigenspaces are  $E^0_n = Span \{e^{\pm inx} \} $ for $n>0 $ and
$E^0_0 = \{ const\}, \; \; \dim E^0_n = 2 $ for $n>0, $ and $\dim
E^0_0 = 1. $

(b) $ Sp (L^0_{Per^-}) = \{n^2, \; n = 1,3,5, \ldots \};$ its
eigenspaces are  $E^0_n = Span \{e^{\pm inx} \}, $ and $ \dim
E^0_n = 2. $

(c) $ Sp (L^0_{dir}) = \{n^2, \; n \in \mathbb{N} \};$ each
eigenvalue $n^2 $ is simple; a corresponding normalized
eigenfunction is $\sqrt{2} \sin nx.$

 Depending on the boundary conditions,
we consider as our canonical orthogonal normalized basis (o.n.b.)
in $ L^2 ([0,\pi]) $ the system $u_k (x), \;k \in \Gamma_{bc}, $ where
\begin{eqnarray}
\label{0.012}
 \text{if} \; \; bc =Per^+ & \quad u_k =\exp (ikx),& \;
k \in \Gamma_{Per^+} = 2\mathbb{Z};  \\
\label{0.014} \text{if} \;
\; bc =Per^- &\quad u_k =\exp (ikx), & \; k \in \Gamma_{Per^-} =
1+2\mathbb{Z}; \\ \text{if} \; \;
bc = Dir  & \quad  u_k =\sqrt{2} \sin kx,
&\; k \in \Gamma_{Dir} = \mathbb{N}.
\label{0.015}
\end{eqnarray}
Let us notice that  $\{u_k (x), \;k \in \Gamma_{bc}\} $ is a
complete system of unit eigenvectors of the operator $L^0_{bc}.$

We set
\begin{equation}
\label{006} H^1_{Per^+} = \left \{f \in H^1: \quad f(\pi) = f(0) \right \},
\quad H^1_{Per^-} = \left \{f \in H^1: \quad f(\pi) = - f(0) \right \}
\end{equation}
and
\begin{equation}
\label{007} H^1_{Dir} = \left \{f \in H^1: \quad f(\pi) = f(0) = 0 \right \}.
\end{equation}
One can easily see that $\{e^{ikx}, \; k \in 2\mathbb{Z} \} $ is
an orthogonal basis in $H^1_{Per^+},$ $\{e^{ikx}, \; k \in 1+
2\mathbb{Z} \} $ is an orthogonal basis in $H^1_{Per^-},$ and
$\{\sqrt{2} sin kx, \; k \in \mathbb{N} \} $ is an orthogonal
basis in $H^1_{Dir}.$

From here it follows that
\begin{equation}
\label{009} H^1_{bc} = \left \{ f(x) = \sum_{k\in \Gamma_{bc}} f_k
u_k (x) \;: \quad \|f\|_{H^1} =\sum_{k\in \Gamma_{bc}}
(1+k^2)|f_k|^2 < \infty \right \}.
\end{equation}

The following statement is well known.

\begin{Lemma}
\label{lem002} (a) If $f,g \in L^1 ([0,\pi]) $ and $ f \sim
\sum_{k \in 2\mathbb{Z}} f_k e^{ikx}, \quad g \sim \sum_{k \in
2\mathbb{Z}} g_k e^{ikx} $ are their Fourier series respect to the
system $\{ e^{ikx},k \in 2\mathbb{Z} \}, $ then the following
conditions are equivalent:

(i) $ {} \quad f$ is absolutely continuous, $ f(\pi) = f(0) $ and
$ f^\prime (x) = g(x) $ a.e.;

(ii)   $ {} \quad g_k = ik f_k \quad \forall k \in 2\mathbb{Z}. $\vspace{3mm}

(b) If $f,g \in L^1 ([0,\pi]) $ and $ f \sim \sum_{k \in 1+
2\mathbb{Z}} f_k e^{ikx}, \quad g \sim \sum_{k \in 1+ 2\mathbb{Z}}
g_k e^{ikx} $ are their Fourier series respect to the system $\{
e^{ikx},k \in 1 + 2\mathbb{Z} \}, $ then the following conditions
are equivalent:

$(i^*) \quad f $ is absolutely continuous, $ f(\pi) = - f(0) $ and
$ f^\prime (x) = g(x) $ a.e.;

$ (ii^*)      \quad g_k = ik f_k \quad \forall k \in 1
+2\mathbb{Z}.
$
\end{Lemma}

\begin{proof}
An integration by parts gives the implication (i) $\Rightarrow $
(ii)  [or $(i^*) \Rightarrow  (ii^*)$].

To prove that (ii) $\Rightarrow $ (i) we set $ G(x) = \int_0^x
g(t) dt. $ By (ii) for $k=0,$ we have $G(\pi) = \int_0^\pi g(t)
dt = \pi g_0 = 0. $ Therefore, integrating by parts we get $$ g_k
= \frac{1}{\pi} \int_0^\pi g(x) e^{-ikx} dx =\frac{1}{\pi}
\int_0^\pi e^{-ikx} d
G(x) = ik G_k, $$ where $G_k =\frac{1}{\pi} \int_0^\pi e^{-ikx}
G(x) dx $ is the $k$--th Fourier coefficient of $G.$ Thus, by
(ii), we have $G_k = f_k $ for $k \neq 0,$ so by the Uniqueness
Theorem for Fourier series $ f(x) = G(x) + const, $ i.e.(i) holds.

Finally, the proof of the implication $(ii^*) \Rightarrow (i^*) $
could be reduced to part (a) by considering the functions
$\tilde{f} (x) = f(x) e^{ix} \sim \sum_{k \in 1+2\mathbb{Z}}
f_{k-1} e^{ikx}$ and $ \tilde{g} (x) = g(x) e^{ix} + i f(x)
e^{ix}.$ We omit the details.
\end{proof}

The next proposition gives the Fourier representations of the
operators $L_{Per^\pm} $ and their domains.

\begin{Proposition}
\label{prop001} In the above notations,
if $y \in H^1_{Per^\pm}, $ then
we have $y=\sum_{\Gamma_{Per^\pm}} y_k e^{ikx} \in
D(L_{Per^\pm})$  and  $ \ell ( y) =  h = \sum_{\Gamma_{Per^\pm}} h_k  e^{ikx}  \in H^0$
if and only if
\begin{equation}
\label{0011}  h_k = h_k (y) : =  k^2 y_k + \sum_{m\in
\Gamma_{Per\pm}} V(k-m) y_m + C y_k ,  \quad   \sum |h_k|^2 <
\infty,
\end{equation}
i.e.,
\begin{equation}
\label{0012} D(L_{Per^\pm}) = \left \{y \in H^1_{Per^\pm}:  \quad
(h_k (y))_{k \in  \Gamma_{Per^\pm}} \in \ell^2 \left
(\Gamma_{Per^\pm} \right )
\right \}
\end{equation}
and
\begin{equation}
\label{0013} L_{Per^\pm} (y)  = \sum_{k \in \Gamma_{Per^\pm}} h_k
(y) e^{ikx}.
\end{equation}

\end{Proposition}

\begin{proof}
Since the proof is the same in the  periodic and  anti--periodic cases, we
consider only the case of periodic boundary conditions.
By (\ref{002}),  if $ y \in
D(L_{Per^+}), $ then $ y \in H^1_{Per^+} $ and
\begin{equation}
\label{0014} \ell (y) = - z^\prime - Q y^\prime + C y = h \in L^2
([0, \pi]),
\end{equation}
where
\begin{equation}
\label{0015}  z := y^\prime - Q y \in W^1_1 ([0,\pi]), \quad
z(\pi) = z(0).
\end{equation}
Let $$y(x) = \sum_{k \in 2\mathbb{Z}} y_k e^{ikx}, \quad z(x) =
\sum_{k \in 2\mathbb{Z}} z_k e^{ikx}, \quad h(x) = \sum_{k \in
2\mathbb{Z}} h_k e^{ikx} \ $$ be the Fourier series of $y, z $ and
$h.$ Since $z(\pi) = z(0), $  Lemma~\ref{lem002} says that the
Fourier series of $z^\prime $ may be obtained by differentiating
term by term the Fourier series of $z,$ and the same property is
shared by $ y $ as a function in $H^1_{Per^+}.$ Thus, (\ref{0014})
implies
\begin{equation}
\label{0016}  -ik z_k - \sum_m q(k-m) im y_m  + C y_k = h_k.
\end{equation}

On the other hand, by (\ref{0015}), we have $ z_k = ik y_k -
\sum_m q(k-m) y_m, $ so substituting that in (\ref{0016}) we get
\begin{equation}
\label{0017} -ik \left [ik y_k - \sum_m q(k-m) y_m \right ] -
\sum_m q(k-m) im y_m  + C y_k = h_k,
\end{equation}
which leads to (\ref{0011})
because $ V(m) = im q(m), \; m \in 2\mathbb{Z}.$

Conversely, if (\ref{0011}) holds, then we have (\ref{0017}).
Therefore,  (\ref{0016}) holds with $ z_k = iky_k - \sum_m q(k-m)
y_m. $

Since  $ y = \sum y_k e^{ikx} \in H^1_{Per^+}, $ the Fourier
coefficients of its derivative are $iky_k, \; k \in 2 \mathbb{Z}
.$ Thus, $(z_k) $ is the sequence of Fourier coefficients of
the function  $z= y^\prime - Qy \in L^1 ([0,\pi]).$

On the other hand, by (\ref{0016}), $(ikz_k)$ is the sequence of
Fourier coefficients of an $L^1 ([0,\pi])$--function.
Therefore, by Lemma  \ref{lem002}, the function $z$ is absolutely
continuous, $z(\pi) = z(0), $ and $(ikz_k) $ is the sequence of
Fourier coefficients of its derivative $z^\prime. $ Thus,
(\ref{0014}) and (\ref{0015}) hold, i.e.,
$y \in D(L_{Per^+}) $ and $ L_{Per^+} y = \ell (y) = h. $
\end{proof}

Now, we are ready to explain the Fourier method for studying the
spectra of the operators $L_{Per^\pm}.$ Let $$ \mathcal{F} : H^0
\to \ell^2 (\Gamma_{Per^\pm}) $$ be the Fourier isomorphisms
defined by corresponding to each function $f \in H^0 $ the
sequence $(f_k) $ of its Fourier coefficients $ f_k  = (f,u_k),$
where $\{u_k, \; k \in \Gamma_{Per^\pm}\} $  is, respectively, the
basis (\ref{0.012}) or (\ref{0.014}). Let $\mathcal{F}^{-1} $ be
the inverse Fourier isomorphism.

Consider the unbounded operators $\mathcal{L}_+ $  and
$\mathcal{L}_- $ acting, respectively, in $\ell^2
(\Gamma_{Per^\pm}) $ as
\begin{equation}
\label{0031} \mathcal{L}_\pm (z) = \left ( h_k (z)  \right )_{k
\in \Gamma_{Per^\pm}}, \quad h_k (z) = k^2 z_k + \sum_{m \in
\Gamma_{Per^\pm}} V(k-m) z_m +C z_k,
\end{equation}
respectively, in the domains
\begin{equation}
\label{0032} D(\mathcal{L}_\pm ) = \left \{ z \in \ell^2 (|k|,
\Gamma_{Per^\pm}): \; \;\mathcal{L}_\pm (z) \in \ell^2
(\Gamma_{Per^\pm})  \right \},
\end{equation}
where $\ell^2 (|k|, \Gamma_{Per^\pm})$ is the weighted
$\ell^2$--space $$ \ell^2 (|k|, \Gamma_{Per^\pm})= \left
\{z=(z_k)_{k \in \Gamma_{Per^\pm}}: \;\; \sum_{k} (1+|k|^2)
|z_k|^2 < \infty \right \}.$$

 In view of (\ref{009}) and Proposition \ref{prop001}, the
following theorem holds.
\begin{Theorem}
\label{thm001} In the above notations, we have
\begin{equation}
\label{0033} D(L_{Per^\pm})=\mathcal{F}^{-1} \left (
D(\mathcal{L}_\pm ) \right )
\end{equation}
and
\begin{equation}
\label{0034} L_{Per^\pm} = \mathcal{F}^{-1} \circ
 \mathcal{L}_\pm \circ \mathcal{F}.
\end{equation}
\end{Theorem}

If it does not lead to confusion, for convenience we will loosely
use one and the same notation $L_{Per^\pm}$  for the operators
$L_{Per^\pm}$ and $\mathcal{L}_\pm. $

\section{Fourier representation for the
Hill--Schr\"odinger
operator with Dirichlet boundary conditions}

In this section we study the Hill--Schr\"odinger operator $L_{Dir}
(v), \; v=C + Q^\prime, $ generated by the differential expression
$\ell_Q (y) = -(y^\prime - Qy)^\prime -Qy^\prime $ considered on
the interval $[0,\pi]$ with Dirichlet boundary conditions $$ Dir:
\quad  y(0) = y(\pi) = 0.$$ Its domain is
\begin{equation}
\label{2.0} D(L_{Dir} (v)) = \left \{ y \in H^1 : \quad y^\prime -
Qy \in W^1_1 ([0,\pi]), \;\; y(0) = y(\pi) = 0, \;\; \ell_Q (y)
\in H^0 \right \},
\end{equation}
and we set
\begin{equation}
\label{2.21} L_{Dir} (v) y= \ell_Q (y).
\end{equation}

\begin{Lemma}
\label{lem2.02} Let $ \begin{pmatrix} y_1 \\ u_1     \end{pmatrix}
$ and $ \begin{pmatrix} y_2 \\ u_2     \end{pmatrix} $ be the
solutions of the homogeneous system (\ref{0.32}) which satisfy
\begin{equation}
\label{2.70}
\begin{pmatrix} y_1 (0) \\ u_1 (0)     \end{pmatrix}
=  \begin{pmatrix} 1 \\ 0     \end{pmatrix}, \qquad
\begin{pmatrix} y_2 (0) \\ u_2 (0)     \end{pmatrix}
=  \begin{pmatrix} 0 \\ 1     \end{pmatrix}.
\end{equation}
 If
\begin{equation}
\label{2.73} y_2 (\pi) \neq 0,
\end{equation}
then the non--homogeneous system (\ref{0.31}) has, for each $f \in
H^0,$ a unique solution $(y,u) = (R_1 (f), R_2 (f))$ such that
\begin{equation}
\label{2.74} y(0) = 0, \quad y (\pi) = 0.
\end{equation}
Moreover, $R_1  $ is a linear continuous operator from $H^0$ into
$H^1, $ and $R_2 $ is a linear continuous operator in $H^0 $ with
a range in $W^1_1 ([0,\pi]).$
\end{Lemma}

\begin{proof}
By the variation of parameters method, every solution of the
non--homogeneous system (\ref{0.31}) has the form $$
\begin{pmatrix} y (x) \\ u (x)     \end{pmatrix}
= v_1 (x)
\begin{pmatrix} y_1 (x) \\ u_1 (x)     \end{pmatrix}
+ v_2 (x) \begin{pmatrix} y_2 (x) \\ u_2 (x)     \end{pmatrix}, $$
where
\begin{equation}
\label{2.76} v_1 (x) = -\int_0^x  y_2 (x) f(t)dt +  C_1,     \quad
v_2 (x) = \int_0^x  y_1 (x) f(t) dt +  C_2.
\end{equation}
By (\ref{2.70}), the condition $y(0) = 0 $ will be satisfied if
and only if $C_1 = 0.$ If so, the second condition $ y(\pi) =0 $
in (\ref{2.74}) is equivalent to $$ m_1 (f) y_1 (\pi) + (m_2 (f) +
C_2 ) y_2 (\pi) = 0, $$ where $$ m_1 (f) =-\int_0^\pi  y_2 (x)
f(t)dt, \quad m_2 (f) = \int_0^\pi  y_1 (x) f(t) dt.$$ Thus, if
$y_2 (\pi) \neq 0, $ then we have unique solution $ (y,u)$ of
(\ref{0.31}) that satisfies (\ref{2.74}), and it is given by
(\ref{2.76}) with $C_1 =0 $ and
\begin{equation}
\label{2.78} C_2 (f) = - \frac{y_1 (\pi)}{y_2 (\pi)} m_1 (f) - m_2
(f).
\end{equation}
Thus, we have $\begin{pmatrix} y (x) \\ u (x)     \end{pmatrix}
=\begin{pmatrix} R_1 (f) \\ R_2 (f)     \end{pmatrix}, $ where $$
R_1 (f) = \left ( -\int_0^x  y_2 (x) f(t)dt \right ) \cdot y_1 (x)
+ \left ( \int_0^x  y_1 (x) f(t)dt  + C_2 (f) \right ) \cdot y_2
(x) $$ and $$R_2 (f) = \left ( -\int_0^x  y_2 (x) f(t)dt \right )
\cdot u_1 (x) + \left ( \int_0^x  y_1 (x) f(t)dt  + C_2 (f) \right
) \cdot u_2 (x). $$  It is easy to see (compare with the proof of
Lemma~\ref{lem02}) that $ R_1 $  is a linear continuous operator
from $H^0$ into $H^1, $ and $R_2 $ is a linear continuous operator
in $H^0 $ with a range in $W^1_1 ([0,\pi]).$ We omit the details.
\end{proof}

Now, let us consider the systems (\ref{0.31}) and (\ref{0.32})
with a spectral parameter $\lambda $ by setting $C= -\lambda $
there, and let
 $ \begin{pmatrix} y_1 (x, \lambda) \\ u_1 (x,\lambda)     \end{pmatrix} $
and $ \begin{pmatrix} y_2 (x, \lambda) \\ u_2 (x, \lambda)
\end{pmatrix} $ be the solutions of the homogeneous system
(\ref{0.32}) that satisfy (\ref{2.70}) for $x=0. $  Notice that
\begin{equation}
\label{2.79} y_2 (\overline{v}; x, \overline{\lambda} ) =
\overline{y_2 (v;x, \lambda)}.
\end{equation}

\begin{Theorem}
\label{thm2.1} Suppose $ v\in H^{-1}_{loc} (\mathbb{R})$
is $\pi$--periodic.
Then,

(a) the domain $D(L_{Dir} (v)) \in (\ref{2.0}) $ is dense in $H^0;
$

(b) the operator $L_{Dir} (v)$ is closed, and its conjugate
operator is
\begin{equation}
\label{2.24} \left (L_{Dir} (v) \right )^*  = L_{Dir}
(\overline{v}).
\end{equation}
In particular,  if $v$ is real--valued, then the operator $L_{Dir}
(v)$ is self--adjoint.

(c) the spectrum $Sp (L_{Dir} (v))$ of
the operator $L_{Dir} (v) $ is
discrete, and moreover,
\begin{equation}
\label{2.25} Sp (L_{Dir} (v)) = \{\lambda \in \mathbb{C} \; : \;\;
y_2 (\pi, \lambda) = 0 \}.
\end{equation}
\end{Theorem}

\begin{proof}
Firstly, we show that the operators $L_{Dir} (v) $ and $L_{Dir}
(\overline{v}) $ are formally adjoint, i.e.,
\begin{equation}
\label{2.27} \left (L_{Dir}(v)y, h \right ) = \left (f,
L_{Dir}(\overline{v})
 h \right ) \quad \text{if} \;\; y \in
D(L_{Dir}(v)), \;\; h \in D(L_{Dir}(\overline{v})).
\end{equation}

Indeed, in view of (\ref{2.0}),
 we have  $$ \left (L_{Dir}(v) y, h
\right ) =\frac{1}{\pi} \int_0^\pi  \ell_Q (y) \overline{h} dx
=\frac{1}{\pi} \int_0^\pi\left (-(y^\prime -Qy)^\prime
\overline{h} - Qy^\prime \overline{h}\right ) dx $$ $$ =  -
\frac{1}{\pi} (y^\prime -Qy)\overline{h}\text{\huge
$\vert$}_{0}^{\pi} + \frac{1}{\pi} \int_{0}^{\pi} (y^\prime -Qy)
\overline{h^\prime} dx - \int_{0}^{\pi} Qy^\prime \overline{h} dx
$$ $$ = 0 + \frac{1}{\pi} \int_0^\pi  \left ( y^\prime
\overline{h^\prime} - Q y \overline{h^\prime} - Q y^\prime
\overline{h} \right ) dx. $$

The same argument shows that $$ \frac{1}{\pi} \int_0^\pi \left (
y^\prime \overline{h^\prime} - Q y \overline{h^\prime} - Q
y^\prime \overline{h} \right ) dx = \left (y, L_{Dir}
(\overline{v}) h \right ), $$ which completes the proof of
(\ref{2.27}).

Now we apply Lemma \ref{lem001} with $A= L_{Dir} (v)$ and
$B=L_{Dir} (\overline{v}).$  Choose $\lambda \in \mathbb{C} $ so
that $ y_2 (v;\pi, \lambda ) \neq 0 $ (in view of (\ref{00.45}),
see the remark before Theorem~\ref{thm001}, $ y_2 (v; \pi, \lambda
) $ is a non--constant entire function, so such a choice is
possible). Then, in view of (\ref{2.79}), we have $ y_2
(\overline{v}; \pi, \overline{\lambda}) \neq 0 $ also. By
Lemma~\ref{lem2.02},
 $ L_{Dir}(v) - \lambda $ maps bijectively $D(L_{Dir} (v)) $ onto
$H^0$ and $L_{Dir}(\overline{v}) - \overline{\lambda} $ maps
bijectively $D(L_{Dir}(\overline{v})) $ onto $H^0. $ Thus, by
Lemma~\ref{lem001}, $D(L_{Dir} (v))$ is dense in $H^0$ and $\left
( L_{Dir}(v)\right )^* = L_{Dir} (\overline{v}), $ i.e., (a) and
(b) hold. If $ y_2 (v; \pi, \lambda ) = 0, $ then $\lambda $  is
an eigenvalue of the operator $L_{Dir} (v),$ and $y_2 (v;x,
\lambda ) $ is a corresponding eigenvector.
 In view of Lemma~\ref{lem2.02},
this means that that (\ref{2.25}) holds. Since $ y_2 (\pi, \lambda
)$ is a non--constant entire function, the set on the right in
(\ref{2.25}) is discrete. This completes the proof of (c).
\end{proof}

\begin{Lemma}
\label{lem2.1} (a) If $f,g  \in L^1 ([0,\pi])$ and $ f \sim
\sum_{k=1}^\infty f_k \sqrt{2} \sin kx, $  $g \sim g_0 +
\sum_{k=1}^\infty g_k \sqrt{2} \cos kx $ are, respectively,  their
sine and cosine Fourier series, then the following conditions are
equivalent:\\ (i)   $ \quad f$ is absolutely continuous, $ f(0) =
f(\pi) = 0$ and $ g(x) = f^\prime (x) \;$ a.e.; \\ (ii)  $\quad
 g_0 = 0, \;\quad  g_k = k f_k   \quad \forall k\in \mathbb{N}. $

(b) If $f,g  \in L^1 ([0,\pi])$ and $ f \sim f_0 +
\sum_{k=1}^\infty f_k \sqrt{2} \cos kx $ and $g \sim
\sum_{k=1}^\infty g_k \sqrt{2} \sin kx $ are, respectively,  their
cosine  and sine Fourier series, then the following conditions are
equivalent:\\ $(i^*)   \quad f $ is absolutely continuous and $ g(x)
= f^\prime (x) \;$ a.e.;\\ $  (ii^*)   \quad g_k = - k f_k \quad
k\in \mathbb{N}. $
\end{Lemma}

\begin{proof}
(a) We have $(i) \Rightarrow (ii) $ because $ g_0 = \frac{1}{\pi}
\int_0^\pi g(x) dx = \frac{1}{\pi} (f(\pi)-f(0))=0, $ and $$ g_k =
\frac{1}{\pi} \int_0^\pi g(x) \sqrt{2} \cos kx dx = \frac{1}{\pi}
f(x) \sqrt{2} \cos kx \text{\huge $\vert$}_0^\pi  + \frac{k}{\pi}
\int_0^\pi f(x) \sqrt{2} \sin kx dx = k f_k $$ for every $ k \in
\mathbb{N}. $

To prove that $(ii) \Rightarrow (i), $ we set $ G(x) = \int_0^x
g(t) dt ;$ then $G(\pi) = G(0) = 0$ because $g_0 = 0.$ The same
computation as above shows that  $ g_k = k G_k \; \forall k \in
\mathbb{N}, $ so the sine Fourier coefficients of two
$L^1$--functions $G$ and $f$ coincide. Thus, $ G(x) = f(x), $
which completes the proof of (a).

The proof of (b) is omitted because it is similar to the proof of
(a).
\end{proof}

Let
\begin{equation}
\label{2.1}
Q \sim   \sum_{k=1}^\infty \tilde{q} (k) \sqrt{2} \sin kx
\end{equation}
be the sine Fourier expansion of $ Q. $
We set also
\begin{equation}
\label{2.3} \tilde{V} (0) = 0, \qquad
\tilde{V}(k) = k \tilde{q} (k) \quad  \text{for} \; k \in \mathbb{N}.
\end{equation}

\begin{Proposition}
\label{prop2.1} In the above notations, if $ y \in H^1_{Dir}, $
then we have $ y  = \sum_{k=1}^\infty y_k \sin kx \in D(L_{Dir}) $
and $ \ell (y) = h = \sum_{k=1}^\infty h_k \sqrt{2} \sin kx \in
H^0 $ if and only if
\begin{equation}
\label{2.5} h_k = h_k (y) = k^2 y_k + \frac{1}{\sqrt{2}}
\sum_{k=1}^\infty \left ( \tilde{V} (|k-m|) - \tilde{V} (k+m)
\right ) y_m + C y_k,   \quad  \sum |h_k|^2 < \infty,
\end{equation}
i.e.,
\begin{equation}
\label{2.7} D(L_{Dir}) = \left \{ y \in H^1_{Dir} : \quad (h_k
(y)_1^\infty \in \ell^2 (\mathbb{N}) \right \}, \quad
 L_{Dir} (y) = \sum_{k=1}^\infty h_k (y) \sqrt{2} \sin kx.
\end{equation}
\end{Proposition}

\begin{proof}
By (\ref{2.0}), if $y \in D(L_{Dir} ),$ then $y \in H^1_{Dir}$ and
$$ \ell (y) = - z^\prime - Qy^\prime + Cy = h \in L^2 ([0,\pi]),
$$ where
\begin{equation}
\label{2.9} z:= y^\prime - Qy \in W^1_1 ([0,\pi]).
\end{equation}
Let $$ y \sim \sum_{k=1}^\infty y_k \sqrt{2} \sin kx, \quad z \sim
\sum_{k=1}^\infty z_k \sqrt{2} \cos kx, \quad
 h \sim \sum_{k=1}^\infty h_k \sqrt{2} \sin kx
 $$
be the sine series of $y $ and $h,$ and the cosine series of $z.$
Lemma~\ref{lem2.1} yields $$ z^\prime \sim \sum_{k=1}^\infty (-k
z_k) \sqrt{2} \sin kx, \quad y^\prime \sim \sum_{k=1}^\infty k y_k
\sqrt{2} \cos kx. $$ Therefore,
\begin{equation}
\label{2.8}
h_k = k z_k - (Q y^\prime)_k  + Cy_k,
\quad k \in \mathbb{N},
\end{equation}
where $(Q y^\prime)_k$
are the sine coefficients of the
function $Q y^\prime  \in L^1 ([0,\pi]). $

By (\ref{2.9}), we have $$ z_k = ky_k - (Qy)_k, $$ where $(Qy)_k $
is the $k$-th cosine coefficient of $Qy.$ It can be found by the
formula $$ (Qy)_k = \frac{1}{\pi} \int_0^\pi Q(x) y(x) \sqrt{2} \cos
kx dx = \sum_{m=1}^\infty a_m \cdot y_m, $$ with $$a_m  =a_m (k) =
\frac{1}{\pi} \int_0^\pi Q(x) \sqrt{2} \cos kx \sqrt{2} \sin mx dx =
\frac{1}{\pi} \int_0^\pi Q(x) [\sin (m+k)x  + \sin (m-k)x ] dx $$ $$
= \frac{1}{\sqrt{2}} \begin{cases}  \tilde{q}(m+k) + \tilde{q}(m-k),
& m>k\\ \tilde{q}(2k), & m=k\\ \tilde{q}(m+k) - \tilde{q}(k-m) &
m<k.
\end{cases}
$$ Therefore,
\begin{equation}
\label{2.11} (Qy)_k = \frac{1}{\sqrt{2}}\sum_{m=1}^\infty
\tilde{q}(m+k) y_m - \frac{1}{\sqrt{2}}\sum_{m=1}^{k-1}
\tilde{q}(k-m) y_m +\frac{1}{\sqrt{2}} \sum_{m= k+1}^\infty
\tilde{q}(m-k) y_m.
\end{equation}

In an analogous way we can find the sine coefficients of
$Qy^\prime $ by the formula $$ (Qy^\prime)_k = \frac{1}{\pi}
\int_0^\pi Q(x) y^\prime (x) \sqrt{2} \sin kx dx =
\sum_{m=1}^\infty b_m \cdot m y_m, $$ where $b_m $ are the cosine
coefficients of $Q(x) \sqrt{2} \sin kx, $ i.e., $$ b_m =b_m (k)=
\frac{1}{\pi} \int_0^\pi Q(x)\sqrt{2} \sin kx \sqrt{2} \cos mx
=\frac{1}{\pi} \int_0^\pi Q(x) [\sin (k+m)x + \sin (k-m)x ] dx $$
$$=\frac{1}{\sqrt{2}} \begin{cases} \tilde{q}(k+m)
+\tilde{q}(k-m), & m<k,\\ \tilde{q}(2k), & m=k,\\ \tilde{q}(k+m) -
\tilde{q}(m-k) & m>k.
\end{cases}
$$ Thus we get
\begin{equation}
\label{2.12} (Qy^\prime )_k =\frac{1}{\sqrt{2}} \sum_{m=1}^\infty
\tilde{q}(m+k) m y_m + \frac{1}{\sqrt{2}}\sum_{m=1}^{k-1}
\tilde{q}(k-m) my_m -\frac{1}{\sqrt{2}}\sum_{m= k+1}^\infty
\tilde{q}(m-k) m y_m.
\end{equation}
Finally, (\ref{2.11}) and (\ref{2.12}), imply that $$  k^2 y_k - k
(Qy)_k - (Qy^\prime )_k $$ $$ = k^2 y_k - \frac{1}{\sqrt{2}}
\sum_{m=1}^\infty (m+k) \tilde{q} (m+k) +\frac{1}{\sqrt{2}}
\sum_{m=k+1}^\infty (m-k) \tilde{q} (m-k) +\frac{1}{\sqrt{2}}
\sum_{m=1}^{k-1} (k-m) \tilde{q} (k-m). $$ Hence, in view of
(\ref{2.3}), we have $$ h_k = k^2 y_k + \frac{1}{\sqrt{2}}
\sum_{m=1}^\infty \left ( \tilde{V} (|k-m|) - \tilde{V} (k+m)
\right ) y_m + C y_k, $$ i.e.,  (\ref{2.5}) holds.

Conversely, if (\ref{2.5}) holds, then going back we can see, by
(\ref{2.8}), that $z= y^\prime - Q y \in L^2 ([0,\pi]) $ has the
property that $k z_k, \; k \in \mathbb{N}, $ are the sine
coefficients of an $L^1 ([0,\pi])$--function. Therefore, by
Lemma~\ref{lem2.1}, $z$ is absolutely continuous and those numbers
are the sine coefficients of its derivative $z^\prime.$  Hence, $z
= y^\prime - Qy \in W^1_1 ([0,\pi]) $ and $ \ell (y) = h, $ i.e.,
$ y \in D(L_{Dir}) $ and $L_{Dir} (y) = h. $
\end{proof}

Let $$ \mathcal{F} : H^0 \to \ell^2 (\mathcal{N}) $$ be the
Fourier isomorphisms that corresponds to each function $f \in H^0
$ the sequence $(f_k)_{k\in\mathbb{N}} $ of its Fourier
coefficients $ f_k = (f,\sqrt{2} \sin kx),$ and let
$\mathcal{F}^{-1} $ be the inverse Fourier isomorphism.

Consider the unbounded operator $\mathcal{L}_{d} $  and
 acting in $\ell^2 (\mathbb{N}) $ as
\begin{equation}
\label{2.31} \mathcal{L}_{d} (z) = \left ( h_k (z)  \right )_{k
\in \mathbb{N}}, \quad h_k (z) = k^2 z_k + \frac{1}{\sqrt{2}}
\sum_{m \in \mathbb{N}} \left ( \tilde{V}(|k-m|) -\tilde{V} (k+m)
\right )
 z_m +C z_k
\end{equation}
in the domain
\begin{equation}
\label{2.32} D(\mathcal{L}_{d} ) = \left \{ z \in \ell^2 (|k|,
\mathbb{N}): \; \;\mathcal{L}_{d} (z) \in \ell^2 (\mathbb{N})
\right \},
\end{equation}
where $\ell^2 (|k|, \mathbb{N})$ is the weighted $\ell^2$--space
$$ \ell^2 (|k|,\mathbb{N})= \left \{z=(z_k)_{k \in \mathbb{N}}:
\;\; \sum_{k} |k|^2 |z_k|^2 < \infty \right \}.$$

 In view of (\ref{009}) and Proposition \ref{prop2.1}, the
following theorem holds.
\begin{Theorem}
\label{thm2.2} In the above notations, we have
\begin{equation}
\label{2.33} D(L_{Dir})=\mathcal{F}^{-1} \left ( D(\mathcal{L}_{d}
) \right )
\end{equation}
and
\begin{equation}
\label{2.34} L_{Dir} = \mathcal{F}^{-1} \circ
 \mathcal{L}_{d} \circ \mathcal{F}.
\end{equation}
\end{Theorem}

If it does not lead to confusion, for convenience we will loosely
use one and the same notation $L_{Dir}$  for the operators
$L_{Dir}$ and $\mathcal{L}_{d}. $

\section{Localization of spectra}

Throughout this section we need the following
lemmas.

\begin{Lemma}
\label{lem3.1} For each $n \in \mathbb{N} $
\begin{equation}
\label{3.1} \sum_{k \neq \pm n} \frac{1}{|n^2 - k^2|} <
\frac{2\log 6n}{n};
\end{equation}
\begin{equation}
\label{3.2} \sum_{k \neq \pm n} \frac{1}{|n^2 - k^2|^2} <
\frac{4}{n^2}.
\end{equation}
\end{Lemma}

The proof is elementary, and therefore, we omit it.

\begin{Lemma}
\label{lem3.2} There exists an absolute constant $C> 0 $ such that

(a) if $n \in \mathbb{N} $ and $b\geq 2, $ then
\begin{equation}
\label{3.3} \sum_k \frac{1}{|n^2 -k^2|+b} \leq C \frac{\log
b}{\sqrt{b}};
\end{equation}

(b)  if $n \geq 0$ and $b > 0 $  then
\begin{equation}
\label{3.4} \sum_{k \neq \pm n} \frac{1}{|n^2 - k^2|^2+ b^2} \leq
\frac{C}{(n^2+b^2)^{1/2} (n^4 + b^2)^{1/4}}.
\end{equation}
\end{Lemma}

A proof of this lemma can be found in
\cite{DM15}, see Appendix, Lemma~79.

We study the localization of spectra of the operators $L_{Per^\pm}
$ and $L_{Dir} $ by using their Fourier representations. By
(\ref{0031}) and Theorem~\ref{thm001}, each of the operators $L=
L_{Per^\pm}$ has the form
\begin{equation}
\label{3.5} L = L^0 +V,
\end{equation}
where the operators $L^0 $ and $V$ are defined by their action on
the sequence of Fourier coefficients of any $y =
\sum_{\Gamma_{Per^\pm}} y_k \exp {ikx} \in H^1_{Per^\pm}:$
\begin{equation}
\label{3.6} L^0: \; (y_k) \to (k^2 y_k), \quad k \in
\Gamma_{Per^\pm}
\end{equation}
and
\begin{equation}
\label{3.7} V:\; (y_m) \to (z_k), \quad z_k = \sum_{m} V(k-m) y_m,
\quad k, m \in \Gamma_{Per^\pm}.
\end{equation}
(We suppress in the notations of $L^0 $ and $V$ the dependence on
the boundary conditions $Per^\pm.$)

In the case of Dirichlet boundary condition, by  (\ref{2.31}) and
Theorem~\ref{thm2.2}, the operator $L = L_{Dir}$ has the form
(\ref{3.5}), where the operators $L^0 $ and $V$ are defined by
their action on the sequence of Fourier coefficients of any $y =
\sum_{\mathbb{N}} y_k \sqrt{2} \sin kx \in H^1_{Dir}:$
\begin{equation}
\label{3.9} L^0: \; (y_k) \to (k^2 y_k), \quad k \in
\mathbb{N}
\end{equation}
and
\begin{equation}
\label{3.10} V: \; (y_m) \to (z_k), \quad z_k = \frac{1}{\sqrt{2}}
\sum_{m} \left (\tilde{V}(|k-m|) -\tilde{V} (k+m) \right ) y_m,
\quad k, m \in \mathbb{N}.
\end{equation}
(We suppress in the notations of $L^0 $
and $V$ the dependence on
the boundary conditions $Dir.$)

Of course, in the regular case where $ v \in L^2 ([0,\pi]),$ the
operators $L^0$ and $V$ are, respectively, the Fourier
representations of $ -d^2/dx^2 $ and the multiplication operator
$y\to v \cdot y.$ But if   $v \in H^{-1}_{loc} (\mathbb{R})$ is a
singular periodic potential, then the situation is more
complicated, so we are able to write (\ref{3.5}) with (\ref{3.6})
and (\ref{3.7}), or (\ref{3.9}) and (\ref{3.10}), only after
having the results from Section 3 and 4 (see Theorem~\ref{thm001}
and Theorem~\ref{thm2.2}).

 In view of (\ref{3.6}) and (\ref{3.9}) the operator $L^0 $
 is diagonal,
 so, for $\lambda \neq k^2, \; k \in \Gamma_{bc}, $
 we may consider (in the space $\ell^2 (\Gamma_{bc})$)
its inverse operator
\begin{equation}
\label{3.11} R^0_\lambda :\; (z_k) \to \left ( \frac{z_k}{\lambda-
k^2} \right ),
 \qquad k \in \Gamma_{bc}.
\end{equation}

One of the technical difficulties that arises for singular
potentials is connected with the standard perturbation type
formulae for the resolvent $R_\lambda = (\lambda - L^0 -V)^{-1}.$
In the case where $v \in L^2 ([0,\pi]) $ one can represent the
resolvent in the form (e.g., see \cite{DM15}, Section 1.2)
\begin{equation}
\label{3.21} R_\lambda = (1-R_\lambda^0 V)^{-1} R_\lambda^0 =
\sum_{k=0}^\infty (R_\lambda^0 V)^k R_\lambda^0,
\end{equation}
or
\begin{equation}
\label{3.22} R_\lambda = R_\lambda^0 (1-V R_\lambda^0 )^{-1}  =
\sum_{k=0}^\infty R_\lambda^0 (V R_\lambda^0)^k.
\end{equation}
 The simplest
conditions that guarantee the convergence of the series
(\ref{3.21}) or (\ref{3.22}) in $\ell^2 $ are $$ \|R_\lambda^0 V\|
<1, \quad \text{respectively,} \quad \|V R_\lambda^0\| < 1. $$
Each of these conditions can be easily verified for large enough
$n$ if $ Re \, \lambda \in [n-1, n+1]$ and $ |\lambda - n^2 | \geq
C(\|v\|), $ which leads to a series of results on the spectra,
zones of instability and spectral decompositions.

The situation is more complicated if $v$ is a singular potential.
Then, in general, there are no good estimates for the norms of
 $      R^0_\lambda V  $ and
$ V R^0_\lambda. $  However, one can write (\ref{3.21}) or
(\ref{3.22}) as
\begin{equation}
\label{3.23} R_\lambda = R^0_\lambda + R^0_\lambda V R^0_\lambda +
R^0_\lambda V R^0_\lambda V R^0_\lambda + \cdots = K^2_\lambda +
\sum_{m=1}^\infty K_\lambda(K_\lambda V K_\lambda)^m K_\lambda,
\end{equation}
provided
\begin{equation}
\label{3.24} (K_\lambda)^2 = R^0_\lambda.
\end{equation}
We define an operator $K= K_\lambda $ with the property
(\ref{3.24})  by its matrix representation
\begin{equation}
\label{3.25} K_{jm} = \frac{1}{(\lambda - j^2)^{1/2}}
\delta_{jm},\qquad j,m \in \Gamma_{bc},
\end{equation}
where $$z^{1/2} = \sqrt{r} e^{i\varphi/2} \quad \text{if} \quad z=
re^{i\varphi}, \;\; 0\leq \varphi < 2\pi. $$

Then $R_\lambda $ is well--defined if
\begin{equation}
\label{3.26} \|K_\lambda V K_\lambda:  \;
\ell^2 (\Gamma_{bc})  \to  \ell^2 (\Gamma_{bc})\| <
1.
\end{equation}

In view of (\ref{0.15}), (\ref{3.7}) and (\ref{3.25}), the matrix
representation of $KVK$ for periodic or anti--periodic boundary
conditions $bc = Per^\pm$ is
\begin{equation}
\label{3.29} (KVK)_{jm} = \frac{V(j-m)}{(\lambda -
j^2)^{1/2}(\lambda - m^2)^{1/2}} =\frac{i(j-m)q(j-m)}{(\lambda -
j^2)^{1/2}(\lambda - m^2)^{1/2}},
\end{equation}
where $j,m \in 2\mathbb{Z} $ for $bc = Per^+,$ and $j,m \in 1+
2\mathbb{Z} $ for $bc=Per^-.$ Therefore, we have for its
Hilbert--Schmidt norm (which majorizes its $\ell^2 $-norm)
\begin{equation}
\label{3.30} \|KVK\|_{HS}^2 = \sum_{j,m \in \Gamma_{Per^\pm}}
\frac{(j-m)^2 |q(j-m)|^2} {|\lambda - j^2| |\lambda - m^2|}.
\end{equation}

By (\ref{2.3}), (\ref{3.10}) and (\ref{3.25}), the matrix
representation of $KVK$ for Dirichlet boundary conditions $bc =
Dir $ is
\begin{equation}
\label{3.32} (KVK)_{jm} = \frac{1}{\sqrt{2}}
\frac{\tilde{V}(|j-m|)}{(\lambda -
j^2)^{1/2}(\lambda - m^2)^{1/2}}
- \frac{1}{\sqrt{2}}
\frac{\tilde{V}(j+m)}{(\lambda -
j^2)^{1/2}(\lambda - m^2)^{1/2}}
\end{equation}
$$ = \frac{1}{\sqrt{2}} \frac{|j-m| \tilde{q}(|j-m|)}{(\lambda -
j^2)^{1/2}(\lambda - m^2)^{1/2}} - \frac{1}{\sqrt{2}} \frac{(j+m)
\tilde{q}(j+m)}{(\lambda - j^2)^{1/2}(\lambda - m^2)^{1/2}}. $$
where $j,m \in \mathbb{N}. $ Therefore, we have for its
Hilbert--Schmidt norm (which majorizes its $\ell^2 $-norm)
\begin{equation}
\label{3.33} \|KVK\|_{HS}^2 \leq 2 \sum_{j,m \in \mathbb{N}}
\frac{(j-m)^2 |\tilde{q}(|j-m|)|^2} {|\lambda - j^2| |\lambda -
m^2|} + 2 \sum_{j,m \in \mathbb{N}} \frac{(j+m)^2
|\tilde{q}(j+m)|^2} {|\lambda - j^2| |\lambda - m^2|}.
\end{equation}

We set for convenience
\begin{equation}
\label{3.35}
\tilde{q} (0) = 0, \quad
\tilde{r} (s) = \tilde{q} (|s|) \quad \text{for} \;\; s \neq 0,
\quad s \in \mathbb{Z} .
\end{equation}
In view of (\ref{3.33}) and (\ref{3.35}), we have
\begin{equation}
\label{3.37} \|KVK\|_{HS}^2 \leq
\sum_{j,m \in \mathbb{Z}}
\frac{(j-m)^2 |\tilde{r}(j-m)|^2} {|\lambda - j^2| |\lambda - m^2|}.
\end{equation}

We divide the plane $\mathbb{C}$ into strips, correspondingly to
the boundary conditions,  as follows:

if $bc = Per^+ $ then $\mathbb{C} = H_0 \cup H_2 \cup H_4 \cup
\cdots, $ and

if $bc = Per^- $ then $\mathbb{C} = H_1 \cup H_3 \cup H_5 \cup
\cdots, $\\ where
\begin{equation}
\label{1.6}
H_0 = \{\lambda \in \mathbb{C}: Re \, \lambda \leq 1\}, \quad
H_1 = \{\lambda \in \mathbb{C}: Re \, \lambda \leq 4\},
\end{equation}
\begin{equation}
\label{1.7}
H_n = \{\lambda \in \mathbb{C}: \; (n-1)^2 \leq Re\, \lambda
\leq (n+1)^2 \},  \quad n \geq 2;
\end{equation}

- if $bc = Dir,$  then
$\mathbb{C} = G_1 \cup G_2 \cup G_3 \cup \cdots, $  where
\begin{equation}
\label{1.8}
G_1 = \{\lambda : Re \, \lambda \leq 2\}, \quad
G_n = \{\lambda :  (n-1)n \leq  Re \, \lambda \leq n(n+1) \},
\quad n \geq 2.
\end{equation}

Consider also the discs
\begin{equation}
\label{1.10}
D_n  = \{\lambda \in \mathbb{C}: \; |\lambda - n^2| < n/4 \},
 \quad n \in \mathbb{N},
\end{equation}
Then, for $n \geq 3,$
\begin{equation}
\label{1.12}
\sum_{k \in n + 2\mathbb{Z}} \frac{1}{|\lambda - k^2|} \leq
C_1 \frac{\log n}{n}, \quad
\sum_{k \in n + 2\mathbb{Z}} \frac{1}{|\lambda - k^2|^2} \leq
\frac{C_1}{n^2},
\quad  \forall \lambda \in H_n \setminus D_n,
\end{equation}
and
\begin{equation}
\label{1.14}
\sum_{k \in \mathbb{Z}} \frac{1}{|\lambda - k^2|} \leq
C_1 \frac{\log n}{n}, \quad
\sum_{k \in \mathbb{Z}} \frac{1}{|\lambda - k^2|^2} \leq
\frac{C_1}{n^2},
\quad \forall \lambda \in G_n \setminus D_n,
\end{equation}
where $C_1$  is an absolute constant.

Indeed, if $\lambda \in H_n, $ then one can easily see that $$
|\lambda - k^2|  \geq |n^2 - k^2 |/4 \quad \text{for} \quad  k \in
n + 2\mathbb{Z}. $$ Therefore,  if $\lambda \in H_n \setminus D_n,
$  then (\ref{3.1}) implies that $$ \sum_{k \in n + 2\mathbb{Z}}
\frac{1}{|\lambda - k^2|} \leq \frac{2}{n/4} + \sum_{k\neq \pm n}
\frac{4}{|n^2 - k^2|} \leq \frac{8}{n} + \frac{8\log 6n}{n} \leq
C_1 \frac{\log n}{n}, $$ which proves the first inequality in
(\ref{1.12}). The second inequality in (\ref{1.12}) and the
inequalities in (\ref{1.14}) follow from Lemma~\ref{lem3.1} by the
same argument.

Next we estimate the Hilbert--Schmidt norm of the operator
$K_\lambda V K_\lambda $ for $bc = Per^\pm $ or $Dir,$  and
correspondingly, $ \lambda \in H_n \setminus D_n $  or $ \lambda
\in G_n \setminus D_n, \; n \in \mathbb{N}. $

For each $\ell^2$--sequence  $x=(x(j))_{j \in \mathbb{Z}} $ and
$m \in \mathbb{N}$
we set
\begin{equation}
\label{1.15}
\mathcal{E}_m (x) = \left (   \sum_{|j|\geq m} |x(j)|^2   \right )^{1/2}.
\end{equation}

\begin{Lemma}
\label{lem3.3} Let $v = Q^\prime, $ where $Q(x)  = \sum_{k \in
2\mathbb{Z} } q(k) e^{ikx} =\sum_{m=1}^\infty \tilde{q}(m) \sqrt{2}
\sin mx $ is a $\pi$--periodic  $L^2([0,\pi])$ function, and let $$
q = (q(k) )_{k \in 2\mathbb{Z}}, \quad \tilde{q} = (\tilde{q}
(m))_{m \in \mathbb{N}} $$ be the sequences of its Fourier
coefficients respect to the orthonormal bases $\{e^{ikx}, \; k \in
2\mathbb{Z} \}$ and $\{\sqrt{2} \sin mx, \; m \in \mathbb{N} \}.$
Then, for $n\geq 3,$
\begin{equation}
\label{1.17} \|K_\lambda VK_\lambda \|_{HS} \leq C \left (
\mathcal{E}_{\sqrt{n}} (q) + \|q\|/ \sqrt{n} \right ), \quad
\lambda \in H_n \setminus D_n, \;\; bc = Per^\pm,
\end{equation}
and
\begin{equation}
\label{1.18} \|K_\lambda VK_\lambda \|_{HS} \leq C \left (
\mathcal{E}_{\sqrt{n}} (\tilde{q}) + \|\tilde{q}\|/ \sqrt{n}
\right ), \quad \lambda \in G_n \setminus D_n,  \;\; bc = Dir,
\end{equation}
where $C$ is an absolute constant.
\end{Lemma}

\begin{proof}

Fix $n \in \mathbb{N}.$
We prove only (\ref{1.17}) because,
in view of (\ref{3.35}) and (\ref{3.37}),
the proof of (\ref{1.18}) is practically the same
(the only difference is that the summation indices
will run in $\mathbb{Z}$).

By (\ref{3.30}),
\begin{equation}
\label{1.21} \|KVK\|^2_{HS} \leq \sum_s \left ( \sum_m
\frac{s^2}{|\lambda -m^2||\lambda - (m+s)^2|} \right )  |q(s)|^2 =
\Sigma_1 + \Sigma_2 + \Sigma_3,
\end{equation}
where $s \in  2\mathbb{Z}, \;  m \in n + 2\mathbb{Z} $ and
\begin{equation}
\label{1.22}
\Sigma_1 = \sum_{|s| \leq \sqrt{n}} \cdots, \quad
\Sigma_2 = \sum_{\sqrt{n} < |s| \leq 4n} \cdots, \quad
\Sigma_3 = \sum_{|s| > 4n} \cdots.
\end{equation}
The Cauchy inequality implies that
\begin{equation}
\label{1.23} \sum_{m\in n + 2\mathbb{Z}} \frac{1}{|\lambda
-m^2||\lambda - (m+s)^2|} \leq \sum_{m\in n + 2\mathbb{Z}}
\frac{1}{|\lambda -m^2|^2}.
\end{equation}
Thus, by (\ref{1.14}) and (\ref{1.15}),
\begin{equation}
\label{1.24} \Sigma_1 \leq \sum_{|s|\leq \sqrt{n}} |q(s)|^2  s^2
\frac{C_1}{n^2} \leq (\sqrt{n})^2 \frac{C_1}{n^2} \|q\|^2 =
\frac{C_1}{n}\|q\|^2, \quad \lambda \in H_n \setminus D_n,
\end{equation}
\begin{equation}
\label{1.25}
\Sigma_2 \leq  (4n)^2 \frac{C_1}{n^2}
 \sum_{|s| > \sqrt{n}} |q(s)|^2
= 16 C_1 \left ( \mathcal{E}_{\sqrt{n}}(q) \right )^2, \quad
\lambda \in H_n \setminus D_n.
\end{equation}

Next we estimate $\Sigma_3 $ for $n \geq 3.$
First we show that if
$|s| > 4n $ then
\begin{equation}
\label{1.26} \sum_m \frac{s^2}{|\lambda -m^2||\lambda - (m+s)^2|}
\leq 16 \frac{C_1 \log n}{n}, \quad \lambda \in H_n \setminus D_n.
\end{equation}
Indeed, if $ |m| \geq |s|/2,$ then  (since $ |s|/4 >n \geq 3 $)
$$|\lambda - m^2| \geq m^2 - |Re \, \lambda | \geq  s^2/4 -
(n+1)^2
> s^2/4 - (|s|/4 +1)^2 \geq s^2/8. $$ Thus, by (\ref{1.12}), $$
\sum_{|m|\geq |s|/2} \frac{s^2}{|\lambda -m^2||\lambda - (m+s)^2|}
\leq \sum_m \frac{8}{|\lambda - (m+s)^2|} \leq 8\frac{C_1 \log
n}{n} $$ for $ \lambda \in H_n \setminus D_n. $ If $|m| < |s|/2,$
then $|m+s| > |s|- |s|/2 = |s|/2 ,$ and therefore, $$ |\lambda
-(m+s)^2| \geq (m+s)^2 - |Re\, \lambda | \geq s^2/4 - (n+1)^2 \geq
s^2/8. $$ Therefore, by (\ref{1.12}), $$ \sum_{|m|<|s|/2}
\frac{s^2}{|\lambda -m^2||\lambda - (m+s)^2|} \leq \sum_m
\frac{8}{|\lambda - m^2|} \leq 8\frac{C_1 \log n}{n} $$ for $
\lambda \in H_n \setminus D_n, $ which proves (\ref{1.26}).

Now, by (\ref{1.26}),
\begin{equation}
\label{1.28} \Sigma_3 \leq 16\frac{C_1 \log n}{n} \sum_{|s| \geq
4n} |q(s)|^2 = 16\frac{C_1 \log n}{n} \left (\mathcal{E}_{4n} (q)
\right )^2.
\end{equation}

Finally, (\ref{1.21}), (\ref{1.24}), (\ref{1.25}) and (\ref{1.28})
imply (\ref{1.17}).

\end{proof}

Let  $H^N $ denote the half--plane
\begin{equation}
\label{1.30} H^N = \{ \lambda \in \mathbb{C}: \;\;  Re\, \lambda <
N^2+N \}, \qquad N \in \mathbb{N},
\end{equation}
and let $R_N $ be the rectangle
\begin{equation}
\label{1.31} R_N = \{\lambda \in \mathbb{C}: \;\; -N < Re\,
\lambda < N^2 +N, \quad |Im \lambda | < N  \}.
\end{equation}

\begin{Lemma}
\label{lem3.4} In the above notations, for $bc= Per^\pm $ or
$Dir,$ we have
\begin{equation}
\label{1.32} \sup \left \{ \|K_\lambda V K_\lambda \|_{HS}, \;\;
\lambda \in H^N \setminus R_N \right \} \leq C \left ( \frac{(\log
N)^{1/2} }{N^{1/4}} \|q\| + \mathcal{E}_{4\sqrt{N}} (q) \right ),
\end{equation}
where $C$ is an absolute constant, and $q$ is replaced by
$\tilde{q}$ if $bc = Dir. $
\end{Lemma}

\begin{proof} Consider the sequence $r= (r(s))_{s \in \mathbb{Z}},$
defined by
\begin{equation}
\label{1.33} r(s) = \begin{cases} 0 & \text{for odd} \; s,  \\
\max(|q(s)|, |q(-s)|) & \text{for even} \; s,
\end{cases}.
\end{equation}
Then, in view
of (\ref{1.33}),  we have $r \in \ell^2 (\mathbb{Z})$ and $\|r\|
\leq 2\|q\|.$

If $bc = Per^\pm, $ then we have,
by (\ref{3.30}),
\begin{equation}
\label{1.34} \|KVK\|_{HS}^2 \leq \sum_{j,m \in \mathbb{Z}}
\frac{(j-m)^2|r(j-m)|^2} {|\lambda - j^2| |\lambda - m^2|}.
\end{equation}
On the other hand, if $bc = Dir,$ then (\ref{3.37}) gives the same
estimate for  $\|KVK\|_{HS}^2$ but with $r$ replaced by the
sequence $\tilde{r} \in (\ref{3.35}).$ So, to prove (\ref{1.32}),
it is enough to estimate the right side of (\ref{1.34}) for
$\lambda \in H^N \setminus R_N \}.$

If $Re\,\lambda  \leq -N,$ then (\ref{1.34}) implies that $$
\|KVK\|_{HS}^2 \leq \sum_{j,m \in \mathbb{Z}}
\frac{(j-m)^2|r(j-m)|^2} {|N + j^2| |N + m^2|}. $$ On the other
hand, for $b\geq 1, $ the following estimate holds:
\begin{equation}
\label{3.43}
\sum_{j,m \in \mathbb{Z}}
\frac{(j-m)^2|r(j-m)|^2} {|b^2 + j^2| |b^2 + m^2|}
\leq 4 \|r\|^2 \frac{1+\pi}{b}.
\end{equation}
Indeed, the left--hand side of (\ref{3.43}) does not exceed
$$
 \sum_{j,m \in \mathbb{Z}} \frac{2(j^2+m^2)|r(j-m)|^2} {|b^2 + j^2|
|b^2 + m^2|}$$   $$ \leq 2\sum_m \frac{1}{|b^2 + m^2|} \sum_j
|r(j-m)|^2 +2\sum_j \frac{1}{|b^2 + j^2|}
\sum_m |r(j-m)|^2 $$  $$
\leq 4 \|r\|^2 \left ( \frac{1}{b^2} + 2 \int_0^\infty
\frac{1}{b^2+x^2} dx \right ) =4 \|r\|^2 \left ( \frac{1}{b^2} +
\frac{\pi}{b} \right ) \leq 4 \|r\|^2 \frac{1+\pi}{b}.            $$
Now, with $b= \sqrt{N},$ (\ref{3.43}) yields
\begin{equation}
\label{1.35} \|KVK\|_{HS}^2 \leq C \frac{\|r\|^2}{\sqrt{N}} \quad
\text{if}
 \;\; Re\,\lambda  \leq -N,
\end{equation}
where $C$ is an absolute constant.

By (\ref{1.34})  and
the elementary inequality
$$|\lambda - m^2| =
\sqrt{(x-m^2)^2 + y^2} \geq (|x-m^2|+|y|)/\sqrt{2},
\quad \lambda = x+iy,$$
we have
\begin{equation}
\label{1.36} \|K_\lambda V K_\lambda \|^2_{HS} \leq \sum_{s \in
\mathbb{Z}} \sigma (x,y;s) |r(s)|^2,
\end{equation}
where
\begin{equation}
\label{1.37} \sigma (x,y;s) = \sum_{m\in \mathbb{Z}}
\frac{2s^2}{(|x-m^2|+|y|)(|x-(m+s)^2|+|y|)}.
\end{equation}

Now, suppose that $ \lambda = x+iy \in H^N \setminus R_N $ and
$|y| \geq N. $ By (\ref{1.8}) and (\ref{1.30}), $$ H^N \subset
\bigcup_{1 \leq n \leq N} G_n,   $$ so $\lambda \in G_n $ for some
$n \leq N.$ Moreover,
\begin{equation}
\label{1.38} \sigma (x,y;s) \leq 16 \sigma (n^2,N;s) \quad
\text{if} \quad \lambda \in G_n, \; \; |y| \geq N.
\end{equation}
Indeed, then one can easily see that $$ |x-m^2|
+|y| \geq \frac{1}{4}(|n^2-m^2| + N), \quad m \in \mathbb{Z}, $$
which implies (\ref{1.38}).

By (\ref{1.37}) and (\ref{1.38}), if $ \lambda = x+iy \in G_n
\setminus R_N  $  and $ |y| \geq N,$ then
\begin{equation}
\label{1.39}
 \|K_\lambda V K_\lambda \|^2_{HS}
\leq \sum_s \sigma (n^2,N;s) |r(s)|^2 \leq  \Sigma_1 + \Sigma_2 +
\Sigma_3,
\end{equation}
where $$ \Sigma_1 =\sum_{|s| \leq 4\sqrt{N}} \sigma (n^2,N;s)
|r(s)|^2, \quad
 \Sigma_2 = \sum_{4\sqrt{N} < |s| \leq 4n}\cdots,
 \quad  \Sigma_3 = \sum_{|s| > 4n} \cdots.
 $$

If $|s| \leq 4\sqrt{N}, $ then the Cauchy inequality and
(\ref{3.4}) imply that $$ \sigma (n^2,N;s) \leq 32N \cdot \sum_{m}
\frac{1}{|n^2-m^2|^2+N^2} \leq 32N \frac{C}{N(n^4 +N^2)^{1/4}}
\leq \frac{32C}{\sqrt{N}}. $$ Thus
\begin{equation}
\label{1.41} \Sigma_1 \leq  \frac{32C}{\sqrt{N}} \|r\|^2.
\end{equation}

If $4\sqrt{N} < |s| \leq 4n $ then the Cauchy inequality and
(\ref{3.4}) yield $$ \sigma (n^2,N;s) \leq 32n^2 \cdot \sum_{m}
\frac{1}{|n^2-m^2|^2+N^2} \leq 32n^2 \frac{C}{N(n^4 +N^2)^{1/4}}
\leq 32C $$ because $n \leq N. $ Thus
\begin{equation}
\label{1.42} \Sigma_2 \leq  32C \cdot \left (
\mathcal{E}_{4\sqrt{N}}(r)\right )^2.
\end{equation}

Let $ |s| > 4n.$ If $|m| < |s|/2 $ then $ |m+s| \geq |s|/2,$ and
therefore, $$ |n^2 - (m+s)^2| \geq |m+s|^2 - n^2 \geq (|s|/2)^2 -
(|s|/4)^2 \geq s^2/8. $$ Thus, by (\ref{3.3}), $$ \sum_{|m| <
|s|/2} \frac{s^2}{(|n^2-m^2|+N)(|n^2 - (m+s)^2|+N)} \leq  \sum_m
\frac{8}{|n^2-m^2|+N} \leq 8C \frac{\log N}{\sqrt{N}}. $$ If $|m|
\geq |s|/2, $  then we have the same estimate because $ m^2 - n^2
\geq (|s|/2)^2 - (|s|/4)^2 \geq s^2/8, $ and therefore, again by
(\ref{3.3}), $$ \sum_{|m| \geq |s|/2}
\frac{s^2}{(|n^2-m^2|+N)(|n^2 - (m+s)^2|+N)} \leq  \sum_m
\frac{8}{|n^2-(m+s)^2|+N} \leq 8C \frac{\log N}{\sqrt{N}}. $$ Thus
$ \sigma (n^2,N;s) \leq 32C  ( \log N)/\sqrt{N}, $ so we have
\begin{equation}
\label{1.43} \Sigma_3 \leq 32C \|r\|^2 \frac{\log N}{\sqrt{N}}.
\end{equation}
Now, in view of (\ref{1.33}) and (\ref{3.35}), the estimates
(\ref{1.35}) and (\ref{1.41})--(\ref{1.43}) yield (\ref{1.32}),
which completes the proof.
\end{proof}

\begin{Theorem}
\label{thm3.1} For each periodic potential $v \in H_{loc}^{-1}
(\mathbb{R}), $ the spectrum of the operators
$L_{bc} (v) $  with $bc = Per^\pm,  \, Dir $
is discrete. Moreover,
if $bc = Per^\pm $ then, respectively,
for each large enough even
number
 $N^+ > 0 $    or odd number $N^-$, we have
\begin{equation}
\label{3.51} Sp \left ( L_{Per^\pm} \right ) \subset R_{N^\pm}
\cup \bigcup_{n \in N^\pm  + 2\mathbb{N}} D_n,
\end{equation}
where $R_N $ is the rectangle (\ref{1.31}),  $D_n = \{\lambda:\;
|\lambda - n^2 | < n/4 \}, $  and
\begin{equation}
\label{3.52} \# \left ( Sp \left ( L_{Per^\pm} \right ) \cap
R_{N^\pm} \right ) =
\begin{cases}
2N^+ +1    \\  2N^-
\end{cases},
\quad
\# \left ( Sp \left ( L_{Per^\pm} \right ) \cap D_n \right ) = 2
\;\;\text{for} \;\; n \in N^\pm  + 2\mathbb{N},
\end{equation}
where each eigenvalue is counted with its algebraic multiplicity.

If $bc = Dir$ then,
for each large enough number
 $N \in \mathbb{N}, $  we have
\begin{equation}
\label{3.53} Sp \left ( L_{Dir} \right ) \subset R_{N} \cup
\bigcup_{n =N+1}^\infty D_n
\end{equation}
and
\begin{equation}
\label{3.54} \# \left ( Sp \left ( L_{Dir} \right ) \cap R_N
\right ) = N+1, \quad \# \left ( Sp \left ( L_{Dir}) \right ) \cap
D_n \right ) = 1 \;\;\text{for} \;\; n > N.
\end{equation}
\end{Theorem}

\begin{proof}
In view of (\ref{3.23}), the resolvent $R_\lambda $ is well
defined if $\|KVK\| <1.$ Therefore, (\ref{3.51}) and (\ref{3.53})
follow from Lemmas \ref{lem3.3} and \ref{lem3.4}.

To prove (\ref{3.52}) and (\ref{3.54}) we use a standard method of
continuous parametrization. Let us consider the one--parameter
family of potentials  $v_\tau (x) = \tau v(x), \; \tau \in [0,1]. $
Then, in the notation of Lemma~\ref{lem3.3}, we have $v_\tau = \tau
\cdot Q^\prime, $ and the assertions of Lemmas \ref{lem3.3} and
\ref{lem3.4} hold with $q $ and $\tilde{q} $ replaced, respectively,
by $ \tau \cdot q $ and $\tau \cdot \tilde{q}.$ Therefore,
(\ref{3.51}) and (\ref{3.53}) hold, with
 $ L_{bc} =L_{bc} (v) $ replaced by $L_{bc} (v_\tau). $
Moreover, the corresponding resolvents $R_\lambda (L_{bc}
(v_\tau)) $ are analytic in $\lambda $ and continuous in $\tau. $

Now, let us prove the first formula in (\ref{3.52}) in the case
$bc= Per^+. $ Fix an even $N^+ \in \mathbb{N} $ so that
(\ref{3.51}) holds, and consider the projection
\begin{equation}
\label{3.56} P^N (\tau) = \frac{1}{2 \pi i} \int_{\lambda \in
\partial R_N} \left (\lambda-L_{Per^+} (v_\tau) \right )^{-1}
d\lambda.
\end{equation}
The dimension $\dim \left (P^N (\tau) \right ) $ gives the number
of eigenvalues inside the rectangle $R_N.$ Being an integer, it is
a constant, so, by the relation (a) at the begging of Section~3,
we have
$$ \dim P^N (1) = \dim P^N (0) =2N^+ +1.$$ In view of the
relations  (a)--(c) at the begging of Section~3, the same argument
shows that (\ref{3.52}) and (\ref{3.54}) hold in all cases.
\end{proof}

{\em Remark.} It is possible to choose the disks $D_n =
\{\lambda:\; |\lambda - n^2 | < r_n \} $ in Lemma~\ref{lem3.3} so
that $ r_n/n \to 0.  $ Indeed, if we take $r_n = n /\varphi (n),$
where $\varphi (n) \to \infty $ but $\varphi (n)/\sqrt{n} \to 0 $
and $\varphi (n) \mathcal{E}_{\sqrt{n}} (W) \to 0, $ then,
modifying the proof of Lemma~\ref{lem3.3}, one can get that
$\|K_\lambda VK_\lambda \|_{HS}  \to 0 $ as $n \to \infty.$
Therefore, Theorem~\ref{thm3.1} could be sharpen: {\em for large
enough $N^\pm $ and $N$, (\ref{3.51})--(\ref{3.54})  hold with
$D_n = \{\lambda:\; |\lambda - n^2 | < r_n \} $ for some sequence
$\{r_n\} $ such that} $ r_n/n \to 0.$

\section{Conclusion}

The main goal of our paper was to bring into the framework of
Fourier method the analysis of Hill--Schr\"odinger operators with
periodic $H^{-1}_{loc} (\mathbb{R})$ potential,
 considered with
periodic, antiperiodic and Dirichlet boundary conditions. As soon
as this is done we can apply the methodology developed in
\cite{KM2,DM3,DM5} (see a detailed exposition in \cite{DM15}) to
study the relationship between smoothness of a potential $v$ and
rates of decay of spectral gaps $ \gamma_n = \lambda^+_n -
\lambda^-_n $ and deviations $\delta_n $ under a weak a priori
assumption $ v \in H^{-1}.$ (In \cite{KM2,DM3,DM5, DM15} the basic
assumption is $ v \in L^2([0,\pi]).$) Still, there is a lot of
technical problems; we present all the details elsewhere. But now
let us give these results as stronger versions of Theorems 54 and
67 in \cite{DM15}.

\begin{Theorem}
\label{thm33.1} Let $L = L^0 + v(x) $ be a Hill--Schr\"odinger
operator with a real--valued $\pi$--periodic potential
$v \in H^{-1}_{loc} (\mathbb{R}) ,$  and
let $\gamma = (\gamma_n)$ be its gap sequence. If $\omega =
(\omega (n))_{n\in \mathbb{Z}} $ is a sub--multiplicative weight
such that
\begin{equation}
\label{33.001}  \frac{\log \omega (n)}{n} \searrow 0 \quad
\text{as} \quad n \to \infty,
\end{equation}
then, with
$$ \Omega = (\Omega (n)), \quad  \Omega (n)=
\frac{\omega (n)}{n}, $$  we have
\begin{equation}
\label{33.002} \gamma \in \ell^2 (\mathbb{N}, \Omega) \Rightarrow
v \in H(\Omega ).
\end{equation}
If $\Omega $ is a sub--multiplicative weight of exponential type,
i.e.,
\begin{equation}
\label{33.1a} \lim_{n\to \infty} \frac{\log \Omega (n)}{n} >0,
\end{equation}
 then there exists $\varepsilon >0 $ such that
\begin{equation}
\label{33.003} \gamma \in \ell^2 (\mathbb{N}, \Omega) \Rightarrow
v \in H(e^{\varepsilon |n|} ).
\end{equation}
\end{Theorem}

The following theorem summarizes our results about the
Hill--Schr\"odinger operator with complex--valued  potentials $v
\in H^{-1}. $

\begin{Theorem}
\label{thm44.2}  Let $L = L^0 + v(x) $ be the Hill--Schr\"odinger
operator with a $\pi$--periodic potential
$v \in H^{-1}_{loc} (\mathbb{R}).$

Then, for large enough $n > N(v)$  the operator $L$ has, in a disc
of center $n^2 $ and radius $r_n = n/4, $ exactly two (counted
with their algebraic multiplicity) periodic (for even $n$), or
antiperiodic (for odd $n$)   eigenvalues $\lambda^+_n $ and $
\lambda^-_n, $ and one Dirichlet eigenvalue $\mu_n. $

Let
\begin{equation}
\label{44.3} \Delta_n = |\lambda^+_n - \lambda^-_n | +
|\lambda^+_n - \mu_n|, \quad n > N (v);
\end{equation}
then, for each sub-multiplicative weight $\omega $ and $$ \Omega =
(\Omega (n)), \quad  \Omega (n)= \frac{\omega (n)}{n}, $$ we have
\begin{equation}
\label{44.4} v \in H(\Omega )  \; \Rightarrow \;  (\Delta_n)  \in
\ell^2 (\Omega ).
\end{equation}

Conversely, in the above notations, if  $\omega = (\omega
(n))_{n\in \mathbb{Z}} $ is a sub--multiplicative weight such that
\begin{equation}
\label{44.5}  \frac{\log \omega (n)}{n} \searrow 0 \quad \text{as}
\quad n \to \infty,
\end{equation}
then
\begin{equation}
\label{44.6}  (\Delta_n) \in \ell^2 (\Omega) \;  \Rightarrow  \; v
\in H(\Omega ).
\end{equation}

If $\omega $ is a sub--multiplicative weight of exponential type,
i.e.,
\begin{equation}
\label{44.7} \lim_{n\to \infty}  \frac{\log \omega (n)}{n} >0
\end{equation}
then
\begin{equation}
\label{44.8} (\Delta_n) \in \ell^2 (\Omega) \; \Rightarrow   \;
\exists  \varepsilon >0: \; v \in H(e^{\varepsilon |n|} ).
\end{equation}
\end{Theorem}

2. Throughout  the paper and in Theorems \ref{thm33.1} and
\ref{thm44.2} we consider three types of boundary conditions:
$Per^\pm $ and $Dir $ in the form ($a^*$) , ($b^*$)
 and ($c^* \equiv c$) adjusted to the differential operators
 (\ref{0.1}) with {\em singular potentials}  $ v \in H^{-1}.$
 It is worth to observe that if $v$ happens to be
 a regular potential, i.e., $v \in L^2 ([0,\pi])$
 (or even $v \in H^\alpha, \, \alpha > -1/2 $)
 the boundary conditions  ($a^*$) and  ($b^*$)
{\em automatically } become equivalent to the boundary conditions
($a$) and ($b$) as we used to write them in the regular case.
Indeed (see the paragraph after (\ref{cor})), we have

($a^*$) $\quad   Per^+: \quad y(\pi)= y(0), \;\left ( y^\prime -
Qy \right ) (\pi)= \left ( y^\prime - Qy  \right ) (0). $

Therefore, with $ v \in L^2, $ both the $L^2$--function $Q$ and the
quasi--derivative $u =  y^\prime - Q y $ are continuous functions,
so the two terms $y^\prime $ and $Q y $  can be considered
separately. Then the second condition in ($a^*$)  can be rewritten
as
\begin{equation}
\label{6.21}
y^\prime (\pi) - y^\prime (0) = Q(\pi) y(\pi) - Q(0) y(0).
\end{equation}
But, since $Q$ is $\pi$--periodic (see Proposition~\ref{prop01}),
\begin{equation}
\label{6.22}
Q(\pi) = Q(0),
\end{equation}
and with the first condition in ($a^*$)
the right side of (\ref{6.21}) is
$Q(0) (y(\pi) - y(0) ) =0. $
Therefore, ($a^*$) comes to the form

($a$)  $\quad y(\pi) = y(0), \quad y^\prime (\pi) = y^\prime (0).
$

Of course, in the same way the condition ($b^*$) {\em
automatically} becomes equivalent to ($b$) if $v \in H^\alpha, \;
\alpha
> -1/2.$

A.Savchuk and A. Shkalikov checked (\cite{SS03}), Theorem 1.5)
which boundary conditions in terms of a function $y$ and its
quasi--derivative $u = y^\prime - Qy$ are regular by
Birkhoff--Tamarkin. Not all of them are reduced to some canonical
boundary conditions in the case of $L^2$--potentials; the result
could depend on the value of $Q(0).$ For example,
Dirichlet--Neumann   bc $$ y(0) =0, \quad  (y^\prime - Q y)(\pi)
=0 $$ would became $$ y(0) =0, \quad y^\prime (\pi) = Q(\pi) \cdot
y(\pi). $$ Of course, one can adjust $Q$ in advance by choosing
(as it is done in \cite{SS05}) $$ Q(x) = - \int_x^\pi v(t) dt
\quad \text{if} \;\; v \in L^2. $$ But this choice is not good if
Dirichlet--Neumann   bc is written with changed roles of the end
points, i.e., $$ (y^\prime - Q y)(0) =0, \quad y(\pi) =0. $$ We
want to restrict ourselves to such boundary conditions with $v \in
H^{-1} $ that if by chance $v \in L^2 $ then the reduced boundary
conditions do not depend on $Q(0).$

We consider as {\em good} self--adjoint bc
only the following ones:
$$ Dir: \quad  y(0) =0, \quad y(\pi) = 0$$
and
$$ y(\pi) = e^{i\theta}  y(0)$$
$$(y^\prime - Q y)(\pi)  =e^{i\theta}
(y^\prime - Q y)(0) + B e^{i\theta} y(0), $$
where $ \theta \in [0, 2\pi) $ and $B $ is real.

Observations of this subsection are quite elementary but they
would be important if we would try to extend statements like
Theorem~\ref{thm44.2} by finding other troikas of boundary
conditions (and corresponding troikas of eigenvalues like
$\{\lambda^+, \lambda^-,   \mu\})$ and using these spectral
triangles and the decay rates of their diameters to characterize a
smoothness of potentials $v$ with a priori assumption $ v \in
H^{-1} $ (or even $v   \in L^2 ([0, \pi]).$

\end{document}